\documentclass[10pt,a4paper]{amsart} 
\usepackage{graphicx,latexsym,amssymb,color}
\theoremstyle{plain} 
\newtheorem{theorem}{Theorem}[section] 
\newtheorem{lemma}[theorem]{Lemma}

\newtheorem*{conjecture}{Conjecture} 
\newtheorem*{question}{Question} 
\theoremstyle{definition}
\newtheorem{definition}{Definition} 
\theoremstyle{remark} 
\newtheorem{remark}[theorem]{Remark} 
\newtheorem{claim}[theorem]{Claim} 
\newtheorem{convention}[theorem]{Convention}

\newtheorem*{acknowledgments}{Acknowledgements} 
\numberwithin{equation}{section}
\numberwithin{figure}{section}
\newcommand{\bd}{\begin{description}}   
\newcommand{\ed}{\end{description}} 
\newcommand{\ba}{\begin{array}}      \newcommand{\ea}{\end{array}} 
\newcommand{\bc}{\begin{center}}     \newcommand{\ec}{\end{center}} 
\newcommand{\be}{\begin{enumerate}}  \newcommand{\ee}{\end{enumerate}} 
\newcommand{\beq}{\begin{eqnarray}}  \newcommand{\eeq}{\end{eqnarray}} 
\newcommand{\beQ}{\begin{eqnarray*}} \newcommand{\eeQ}{\end{eqnarray*}} 
\newcommand{\bi}{\begin{itemize}}    \newcommand{\ei}{\end{itemize}}

\newcommand{\ov}{\overline}

\newcommand{\s}{\sigma} 
\newcommand{\1}{\mathbf{1}}
\newcommand{\n}{ \{ 1,...,n \} }
\newcommand{\ks}{ \{ 1,...,k \} }
\newcommand\modZ {\mathbb{Z}}

\newcommand\modL {\mathcal{L}}
\newcommand{\nbpt}{18}
\newcommand{\figtotext}[3]{\begin{array}{c}\includegraphics{#3}\end{array}}
\newcommand{\double}{\figtotext{\nbpt}{\nbpt}{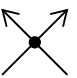}}
\newcommand{\Over}{\figtotext{\nbpt}{\nbpt}{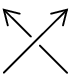}}
\newcommand{\under}{\figtotext{\nbpt}{\nbpt}{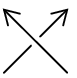}}
\newcommand{\si}{\sigma}
\begin{document} 
\title[Finite Type string link invariants of degree $<5$]{Characterization of Finite Type String Link Invariants of Degree $<5$} 

\author[J.B. Meilhan]{Jean-Baptiste Meilhan} 
\address{Institut Fourier, Universit\'e Grenoble 1 \\
         100 rue des Maths - BP 74\\
         38402 St Martin d'H\`eres , France}
	 \email{jean-baptiste.meilhan@ujf-grenoble.fr}
\author[A. Yasuhara]{Akira Yasuhara} 
\address{Tokyo Gakugei University\\
         Department of Mathematics\\
         Koganeishi \\
         Tokyo 184-8501, Japan}
	 \email{yasuhara@u-gakugei.ac.jp}

\thanks{
The first author is supported by a grant from the Heiwa Nakajima Foundation.  
The second author is partially supported by a Grant-in-Aid for Scientific Research (C) 
($\#$20540065) of the Japan Society for the Promotion of Science.}
%
%
\maketitle
%

\begin{abstract} 
In this paper, we give a complete set of finite type string link invariants of degree $<5$.  
In addition to Milnor invariants, these include several string link invariants constructed by evaluating knot invariants on certain closure of (cabled) string links.  
We show that finite type invariants classify string links up to $C_k$-moves for $k\le 5$, which proves, at low degree, a conjecture due to Goussarov and Habiro.  
We also give a similar characterization of finite type concordance invariants of degree $<6$.  
\end{abstract} 
\section{Introduction}
The notion of Goussarov-Vassiliev finite type link invariants provides a unifying viewpoint on the various quantum link invariants \cite{BNv,Gusarov:91,Gusarov:94,Vassiliev}.
Denote by $\modZ \modL (m)$ the free abelian group generated by the set $\modL(m)$ of isotopy classes of $m$--component oriented, ordered links in $S^3$.  An abelian group-valued link invariant is a finite type invariant of degree $k$ if its linear extension to $\modZ \modL (m)$ vanishes on the $(k+1)$th term of the descending filtration
\begin{equation}
  \label{e7}
  \modZ \modL (m)=J_0(m)\supset J_1(m)\supset \cdots
\end{equation}
where each $J_n(m)$ is generated by certain linear combinations of links associated with singular links with $n$ double points.  See Subsection \ref{fti} for a definition.  

It is a natural question to ask for a topological characterization of finite type invariants.  
Habiro \cite{H} and Goussarov \cite{G} introduced independently the notion of $C_k$-move to address this question.  
A $C_k$-move is a local move on (string) links as illustrated in Figure \ref{cnm}, which can be regarded as a kind of `higher order crossing change' (in particular, a $C_1$-move is a crossing change).   
\begin{figure}[!h]
\includegraphics{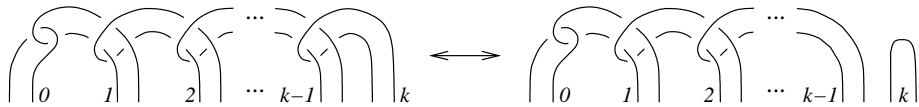}
\caption{A $C_k$-move involves $k+1$ strands of a link, labelled here by integers 
between $0$ and $k$.   } \label{cnm}
\end{figure}

The $C_k$-move generates an equivalence relation on links, called \emph{$C_k$-equivalence}, which becomes finer as $k$ increases.  
This notion can also be defined by using the theory of claspers (see Section \ref{clasp}). 
Goussarov and Habiro showed independently the following.  
\begin{theorem}[\cite{G,H}]\label{cnknots}
Two knots cannot be distinguished by any finite type invariant of order less than $k$ if and only if they are $C_k$-equivalent. 
\end{theorem}  

It is known that the `if' part of the statement holds for links as well, but explicit examples show that the `only if' part of Theorem \ref{cnknots} does not hold for links in general, see \cite[\S 7.2]{H}.    

However, Theorem \ref{cnknots} may generalize to {\em string links}.  Recall that a string link is a proper tangle without closed component (see Subsection \ref{sl} for a precise definition).  
\begin{conjecture}[Goussarov-Habiro ; \cite{G,H}]
Two string links of the same number of components share all finite type invariant of order less than $k$ if and only if they are $C_k$-equivalent.  
\end{conjecture}

One nice property of string links, which suggests some analogy with knots, is that they admit a natural composition.  Indeed the stacking product $\cdot$ endows the set $\mathcal{SL}(n)$ of $n$-string links up to isotopy fixing the endpoints with a structure of monoid.  In particular, string links with $1$ component are exactly equivalent to knots, and their stacking product is equivalent to the connected sum $\sharp$ of knots.   
The Goussarov-Habiro Conjecture is also supported by the fact that there are much more finite type invariants for string links than for links. For example, Milnor invariants \cite{Milnor, Milnor2} are defined for both links and string links, but (except for the linking number) they are of finite type only for string links.  See Subsection \ref{milnor}.  

As in the link case, the `if' part of the conjecture is always true.  
The `only if' part is also true for $k=1$ (in which case the statement is vacuous) and $k=2$ ; the only finite type string link invariant of degree $1$ is the linking number, which is known to classify string links up to $C_2$-equivalence \cite{MN}.  (Note that this actually also applies to links).  
The Goussarov-Habiro conjecture was then (essentially) proved for $k=3$ by the first author in \cite{jbjktr}.  
Massuyeau gave a proof for $k=4$, but it is mostly based on algebraic arguments and thus does not provide any information about the corresponding finite type invariants \cite{massuyeau}.

In this paper, we classify $n$-string links up to $C_k$-move for $k\le 5$, by explicitly giving a complete set of low degree finite type invariants.  In addition to Milnor invariants, these include several `new' string link invariants constructed by evaluating knot invariants on certain closure of (cabled) string links.  See Section \ref{statements} for the statements of these main results.  As a consequence, we prove the Goussarov-Habiro Conjecture for $k\le 5$.  

We also consider the case of finite type concordance invariants.  It is known that, over the rationals, these are all given by Milnor invariants \cite{HMa}.  
We introduce the notion of $C_k$-concordance, which is the equivalence relation on (string) links generated by $C_k$-moves and concordance.  We classify knots up to $C_k$-concordance and show that two $n$-string links ($n\ge 2$) are $C_k$-concordant if and only if they share all finite type concordance invariants of degree $<k$ for $k\le 6$.  (Again, these statement provide a complete set such invariants).  

We also apply some of the techniques developed in this paper to previous works by the authors \cite{yasuhara,MY}.  We first give a classification of $2$-string links up to self $C_3$-moves and concordance, where a self $C_k$-move is a $C_k$-move with all $k+1$ strands in a single component.  Next we consider Brunnian string links.  Recall that a (string) link is Brunnian if it becomes trivial after removing any number of components.  We give a classification of Brunnian $n$-string links up to $C_{n+1}$-equivalence, thus improving a previous results of the authors \cite{MY}.  

The rest of the paper is organized as follows. 
In section 2, we recall the definitions of the main notions of this paper, and state our main results characterizing finite type string link invariants of degree $<5$.  
In section 3, we review the main tool used in the proofs, namely the theory of claspers, and provide several key lemmas.  
The main results are proved in section 4. 
In section 5, we give a characterization of finite type concordance invariants for string links of degree $<6$.  
Finally, we give in Section 6 the classification of $2$-string links up to self $C_3$-moves and concordance, and Section 7 contains our (improved) result on Brunnian string links.  
\begin{acknowledgments}
This work was done while the first author was a visiting Tokyo Institute of Technology.  He thanks Hitoshi Murakami for his warm hospitality.  
\end{acknowledgments}
\section{Statements of the main results}\label{statements}
In this section, we state our main results, which provide a complete set of finite type string link invariants of degree $<5$ and validates the Goussarov-Habiro conjecture up to this degree.
\subsection{Preliminaries} \label{preliminaries}
In this subsection we recall the definitions and properties of finite type string link invariants, and review several examples that will be used in our main results.  
\subsubsection{String links} \label{sl}
Let $n\ge 1$, and let $D^2$ be the standard two-dimensional disk equipped with $n$ marked points $x_1,...,x_n$ in its interior.  
Let $I$ denote the unit interval.  
An \textit{$n$-string link}, or $n$-component string link, is a proper embedding  
\[ \sigma : \bigsqcup_{i=1}^n I_i \rightarrow D^2\times I, \]
of the disjoint union $\sqcup_{i=1}^{n} I_i$ of $n$ copies of $I$ in $D^2\times I$, such that for each $i$, the image $\sigma_i$ of $I_i$ runs from $(x_i,0)$ to $(x_i,1)$.   
Abusing notation, we will also denote by $\s \subset D^2\times I$ the image of the map $\s$, and $\sigma_i$ is called the $i$th string of $\sigma$.  
Note that each string of an $n$-string link is equipped with an (upward) orientation induced by the natural orientation of $I$.

The set $\mathcal{SL}(n)$ of isotopy classes of $n$-string links fixing the endpoints has a monoidal structure, with composition given by the \emph{stacking product} and with the trivial $n$-string link $\1_n$ as unit element.   
We shall sometimes denote the trivial string link by $\1$ when the number of component is irrelevant.  

There is a surjective map $\hat\ :  \mathcal{SL}(n)\rightarrow \mathcal{L}(n)$ which sends an $n$-string link $\s$ to its closure $\hat\s$ (in the usual sense).  For $n=1$, this map is a monoid isomorphism.   

We have a descending filtration 
 \[ \mathcal{SL}(n)=\mathcal{SL}_1(n)\supset \mathcal{SL}_2(n)\supset \mathcal{SL}_3(n)\supset...\] 
where $\mathcal{SL}_k(n)$ denotes the set of {\em $C_k$-trivial} $n$-string links, i.e., string links which are 
$C_k$-equivalent to $\1_n$.  
For $1\le k\le l$, let $\mathcal{SL}_k(n) / C_l$ denote the set of $C_l$-equivalence classes of $C_k$-trivial $n$-string links.  This is known to be a finitely generated nilpotent group.  Furthermore, if $l\le 2k$, this group is abelian \cite[Thm. 5.4]{H}.  
\subsubsection{Finite type invariants of string links}\label{fti}
A \emph{singular $n$-string links} is a proper immersion $\sqcup_{i=1}^{n} I_i \rightarrow D^2\times I$ such that the image of $I_i$ runs from $(x_i,0)$ to $(x_i,1)$ ($1\le i\le n$), and whose singularities are transverse double points (in finite number). 

Denote by $\mathbf{Z}\mathcal{SL}(n)$ the free abelian group generated by $\mathcal{SL}(n)$.  
A singular $n$-string link $\sigma$ with $k$ double points can be expressed as an element of $\mathbf{Z}\mathcal{SL}(n)$ using the following skein formula.  
\begin{equation} \label{double}
\double = \Over - \under 
\end{equation}

Let $A$ be an abelian group. An $n$-string link invariant $f : \mathcal{SL}(n)\rightarrow A$ is a 
\emph{finite type invariant of order $\le k$} if its linear extension to $\mathbf{Z}\mathcal{SL}(n)$ vanishes on every $n$-string-link with (at least) $k+1$ double points.
If $f$ is of order $\le k$ but not of order $k-1$, then $f$ is called a finite type invariant of order $k$.

We recall several classical examples of such invariants in the next two subsections.  

The Kontsevich integral \cite{Kontsevich} 
 $$ Z:\mathcal{SL}(n)\rightarrow \mathcal{A}(\sqcup_n I) $$
is universal among rational-valued finite type string link invariants.  The target space $\mathcal{A}(\sqcup_n I)$ of $Z$ is the space of \emph{Jacobi diagrams} on $\sqcup_{i=1}^{n} I_i$, that is, the vector space over $\mathbb{Q}$ generated by  vertex-oriented unitrivalent diagrams whose univalent vertices are identified with distinct points on $\sqcup_{i=1}^{n} I_i$, modulo the AS, IHX and STU relations \cite{BNv,BN}.  
Recall that $\mathcal{A}(\sqcup_n I)$ is graded by the degree of Jacobi diagrams, which is defined as half the number of vertices.  
\subsubsection{Finite type knot invariants}\label{knotfti}
In this subsection we recall a few classical results on finite type knot invariants.

Recall that the \emph{Conway polynomial} of a knot $K$ has the form 
$$\nabla_K(z)= 1 + \sum_{k\ge 1} a_{2k}(K) z^{2k}. $$
It is not hard to show that the $z^{2k}$-coefficient $a_{2k}$ in the Conway polynomial is a finite type invariant of degree $2k$ \cite{BNv}.  

Recall also that the \emph{HOMFLYPT polynomial} of a knot $K$ is of the form  
$$P(K;t,z)=\sum_{k=0}^N P_{2k}(K;t)z^{2k},$$
where $P_{2k}(K;t)\in \mathbb{Z}[t^{\pm 1}]$ is called the $2k$th coefficient polynomial of $K$.  
Denote by $P_{2k}^{(l)}(K)$ the $l$th derivative of $P_{2k}(K;t)$ evaluated at $t=1$.  
It was proved by Kanenobu and Miyazawa that $P_{2k}^{(l)}$ is a finite type invariant of degree $2k+l$ \cite{KM}.   

Note that both the Conway and HOMFLYPT polynomials of knots are invariant under orientation reversal, and that both are multiplicative under the connected sum \cite{lickorish}.  

In the rest of this paper, we will freely evaluate these invariants on components of an $n$-string link, via the closure  isomorphism $\mathcal{SL}(1)\simeq \mathcal{L}(1)$.  For example, $a_2(\si_i)$ denotes the invariant $a_2$ of the closure $\hat{\si_i}$.  
\subsubsection{Milnor invariants} \label{milnor}
Given an $n$-component oriented, ordered link $L$ in $S^3$, Milnor invariants $\overline{\mu}_L(I)$ of $L$ are defined for each multi-index $I=i_1 i_2 ...i_m$ (i.e., any sequence of possibly repeating indices) among $\n$ \cite{Milnor,Milnor2}.  
The number $m$ is called the \emph{length} of Milnor invariant $\overline{\mu}(I)$, and is denoted by $|I|$.  
Unfortunately, the definition of these $\overline{\mu}(I)$ 
contains a rather intricate self-recurrent indeterminacy.  

Habegger and Lin showed that Milnor invariants are actually well defined integer-valued invariants of string links \cite{HL}, and that the indeterminacy in Milnor invariants of a link is equivalent to the indeterminacy in regarding it as the closure of a string link.  
We refer the reader to \cite{HL} or \cite{yasuhara} for a precise definition of Milnor invariants $\mu(I)$ of string links.  
The smallest length Milnor invariants $\mu_\si(ij)$ of a string link $\si$ coincide with the linking numbers $lk(\hat{\si_i},\hat{\si_j})$.  Milnor invariants are thus sometimes referred to as `higher order linking numbers'.  
 
It is known that $\mu(I)$ is a finite type invariant of degree $|I|-1$ for string links \cite{BN,Lin}.  
\begin{convention} \label{conv_milnor}
As said above, each Milnor invariant $\mu(I)$ for $n$-string links is indexed by a sequence $I$ of \emph{possibly repeating}  integers in $\n$.  In the following, when denoting indices of Milnor invariants, we will always let \emph{distinct} letters denote \emph{distinct} integers, unless otherwise specified.  For example, $\mu(iijk)$ ($1\le i,j,k\le n$) stands for all Milnor invariants $\mu(iijk)$ with $i, j, k \in \n$ pairwise distincts.    
\end{convention}
\subsection{Invariants of degree $\le 2$ for string links} \label{statec3}
We start by recalling the classification of $n$-string links up to $C_3$-equivalence due to the first author.\footnote{Actually, the present statement is stronger than the one appearing in \cite{jbjktr}.  However, the proof given in Subsection \ref{proofc3} is essentially contained in \cite{jbjktr}.  } 

I turns out that, in addition to the $z^2$--coeficient $a_2$ in the Conway polynomial (which is essentially the only finite type knot invariant of degree $\le 2$) and Milnor invariants of length $\le 3$, this classification requires an additional finite type invariant of degree $2$ for $2$-string links 
 $$ f_2:\mathcal{SL}(2)\longrightarrow \mathbb{Z}, $$
defined by $f_2(\sigma)=a_2(\overline{\sigma})$.  Here $\overline{\sigma}$ denotes the \emph{plat closure} of $\sigma$, which is the knot obtained by identifying the two upper (resp. lower) endpoints of $\sigma$.  
More precisely, we have the following.
\begin{theorem}[\cite{jbjktr}]\label{c3}
Let $\si, \si'\in \mathcal{SL}(n)$.  Then the following assertions are equivalent:
\be 
 \item $\si$ and $\si'$ are $C_{3}$-equivalent, 
 \item $\si$ and $\si'$ share all finite type invariants of degree $\le 2$,
 \item $\si$ and $\si'$ have same Kontsevich integral up to degree $2$, 
 \item $\si$ and $\si'$ share all invariants $a_2$ and $f_2$, and all Milnor invariants 
 $\mu(ij)$ ($1\le i<j\le n$) and $\mu(ijk)$ ($1\le i<j<k\le n$).  
\ee 
In $(4)$, by $\si$ and $\si'$ share all invariants $a_2$ and $f_2$, we mean that $\si=\cup_{i=1}^n \si_i$ and $\si'=\cup_{i=1}^n \si'_i$ satisfy $a_2(\si_i)=a_2(\si'_i)$ and $f_2(\si_i\cup \si_j)=f_2(\si'_i\cup \si'_j)$ for all $1\le i<j\le n$. 
\end{theorem}
\begin{remark} \label{rem_thm}
In subsequent statements, we shall make use of a similar abuse of notation as in assertion (4) of Theorem \ref{c3}.    
\end{remark}
%
%
\subsection{Invariants of degree $3$ for string links}\label{statec4}
Recall that there is essentially only one finite type knot invariant of degree $3$, namely $P_0^{(3)}$. 
Let 
$$ f_3:\mathcal{SL}(2)\rightarrow \mathbf{Z} $$ 
be defined by $f_3(\sigma) := P_0^{(3)}\left(\overline{\sigma}\right)$, 
where $\overline{\sigma}$ is the plat-closure of $\sigma$, and let 
$$ V_3:\mathcal{SL}(3)\rightarrow \mathbf{Z} $$ 
be defined by $V_3(\sigma) := P_0^{(3)}\left(cl_3 \sigma\right)$, 
where $cl_3 \sigma$ is the closure operation illustrated in Figure \ref{threeclose}.  

Clearly, $f_3$ and $V_3$ are both finite type invariants of degree $3$.
\begin{theorem}\label{c4}
Let $\si$, $\si'\in \mathcal{SL}(n)$. Then the following assertions are equivalent:
\be 
 \item $\si$ and $\si'$ are $C_{4}$-equivalent, 
 \item $\si$ and $\si'$ share all finite type invariants of degree $\le 3$,
 \item $\si$ and $\si'$ have same Kontsevich integral up to degree $3$, 
 \item $\si$ and $\si'$ share all invariants $a_2$, $P_0^{(3)}$, $f_2$, $f_3$ and $V_3$, and all Milnor invariants 
 $\mu(ij)$, $\mu(iijj)$ ($1\le i<j\le n$), $\mu(ijk)$ ($1\le i<j<k\le n$), $\mu(ijkl)$ ($1\le i,j< k< l\le n$) and $\mu(ijkk)$ ($1\le i,j,k\le n$ ; $i<j$). 
\ee 
\end{theorem}
%
\subsection{Invariants of degree $4$ for string links}\label{statec5}
There are essentially two linearly independent finite type knot invariants of degree $4$, namely $a_4$ and $P_0^{(4)}$.  
We will use these two knot invariants to define a number of finite type string links invariants of degree $4$ by using some cabling and closure operations. 
We start by setting up some notation.  

Given a $3$-string link $\sigma$, denote by $cl_i\si$, $i=0,...,4$, the five knots obtained from $\si$ by taking the closures illustrated in Figure \ref{threeclose}.  
\begin{center}
\begin{figure}[!h]
\includegraphics{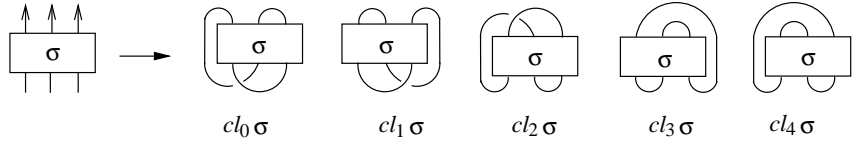}
\caption{The five closures $cl_i \si$ ($i=0,...,4$) of a $3$-string link $\si$.  } \label{threeclose}
\end{figure}
\end{center}
Also, given a $4$-string link $\sigma$, denote by $K_i(\si)$, $i=1,2,3$, the knot obtained by the closure operations represented in Figure \ref{fourclose}.  
\begin{center}
\begin{figure}[!h]
\includegraphics{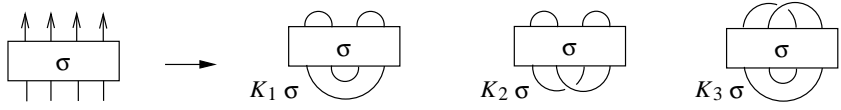}
\caption{The three closures $K_i \si$ ($i=1,2,3$) of a $4$-string link $\si$.  } \label{fourclose}
\end{figure}
\end{center}
Finally, for a $2$-string link $\si$, denote by $\Delta_i\si$ ($i=1,2$) the $3$-string link obtained by taking a $0$-framed parallel copy of the $i$th component $\si_i$ of $\si$.  

We now define five invariants of $2$-string links as follows.  
For $1\le i\le 5$, let 
$$ f^i_4:\mathcal{SL}(2)\rightarrow \mathbf{Z} $$ 
be defined by the following 
\beQ
f^1_4(\sigma) :=a_4\left(\overline{\sigma}\right) & \textrm{ , } & f^2_4(\sigma) := P_0^{(4)}\left(\overline{\sigma}\right), \\
f^3_4(\sigma) :=a_4\left(cl_0(\Delta_1\sigma)\right) & \textrm{ , } & f^4_4(\sigma) := P_0^{(4)}\left(cl_0(\Delta_1\sigma)\right), \\
 & \textrm{and} & f^5_4(\sigma) := P_0^{(4)}\left(cl_1(\Delta_2\sigma)\right).
\eeQ
We have that $f^i_4$ is a finite type invariants of degree $4$ for $i=1,...,5$. (It is immediate for $i=1,2$, and easy to check for $i=3,4,5$.)

Next we define seven invariants of $3$-string links.  
For $1\le i\le 7$, let 
$$ V^i_4:\mathcal{SL}(3)\rightarrow \mathbf{Z} $$ 
be defined by the following 
\beQ
V^1_4(\sigma) :=a_4\left(cl_1\sigma\right) & \textrm{ , } & V^2_4(\sigma) := P_0^{(4)}\left(cl_1\sigma\right), \\
V^3_4(\sigma) :=a_4\left(cl_2\sigma\right) & \textrm{ , } & V^4_4(\sigma) := P_0^{(4)}\left(cl_2\sigma\right), \\
V^5_4(\sigma) :=a_4\left(cl_3\sigma\right) & \textrm{ , } & V^6_4(\sigma) := P_0^{(4)}\left(cl_3\sigma\right), \\
 & \textrm{and} & V^7_4(\sigma) := P_0^{(4)}\left(cl_4\sigma\right). 
\eeQ
Clearly, each $V^i_4$ is a finite type invariant of degree $4$, $i=1,...,7$.

Finally, we define three finite type invariants of degree $4$ of $4$-string links
$$ W^i_4:\mathcal{SL}(4)\rightarrow \mathbf{Z} $$ 
by setting $W^i_4(\si):=P^{(4)}_0\left(K_i(\sigma)\right)$, $1\le i\le 3$.  

These various invariants, together with Milnor invariants of length $\le 5$, give the following classification of $n$-string links up to $C_5$-equivalence.  
\begin{theorem}\label{c5}
Let $\si$, $\si'\in \mathcal{SL}(n)$. Then the following assertions are equivalent:
\be 
 \item $\si$ and $\si'$ are $C_{5}$-equivalent, 
 \item $\si$ and $\si'$ share all finite type invariants of degree $\le 4$,
 \item $\si$ and $\si'$ have same Kontsevich integral up to degree $4$, 
 \item $\si$ and $\si'$ share all knots invariants of degree $\le 4$, all invariants $f_2$, $f_3$, $V_3$, $f^i_4$, $V^i_4$ and $W^i_4$, and all Milnor invariants of length $\le 5$,  
\ee 
where, in $(4)$, $\si$ and $\si'$ share all Milnor invariants of length $\le 5$ if and only if they share all $\mu(ij)$, $\mu(iijj)$ ($1\le i<j\le n$), $\mu(ijk)$ ($1\le i<j<k\le n$), $\mu(ijkl)$ ($1\le i,j< k< l\le n$), $\mu(ijkk)$ ($1\le i,j,k\le n$ ; $i<j$), $\mu(ijklm)$ ($1\le i,j,k<l<m\le n$), $\mu(iiijk)$, $\mu(ijjkk)$ and $\mu(jikll)$ ($1\le i,j,k,l\le n$ ; $j<k$). 
\end{theorem}
%
\begin{remark}
A complete set of finite type link invariant of degree $\le 3$ has been computed in \cite{kmt} using weight systems and chord diagrams.  For $2$-component links, this has been done for degree $\le 4$ invariants in \cite{kanenobu}.  
All invariants are given by coefficients of the Conway and HOMFLYPT polynomials of sublinks.  
\end{remark}
\section{Claspers and local moves on links} \label{clasp}
The main tool in the proofs of our main results is the theory of claspers.  We recall here the main definitions and properties of this theory, and state a couple of additional lemmas that will be useful in later sections.  
\subsection{A brief review of clasper theory} \label{review} 
For convenience, we give all definitions and statements in the context of string links.  For a general definition of claspers, we refer the reader to \cite{H}.  
\begin{definition}\label{defclasp}
Let $\si$ be a string link.  
An embedded surface $G$ is called a {\em graph clasper} for $\si$ if it satisfies the following three conditions:
\be 
\item $G$ is decomposed into disks and bands, called {\em edges}, each of which 
connects two distinct disks.
\item The disks have either 1 or 3 incident edges, called {\em leaves} or 
{\em nodes} respectively.
\item $G$ intersects $\si$ transversely, and the intersections are contained in the union of the interior of the leaves. 
\ee
In particular, if each connected component of $G$ is simply connected, we call it a \emph{tree clasper}. 
\end{definition}
A graph clasper for a string link $\si$ is \emph{simple} if each of its leaves intersects $\si$ at one point.  

The degree of a connected graph clasper $G$ is defined as half of the number of nodes and leaves.  
We call a degree $k$ connected graph clasper a \emph{$C_k$-graph}.  
A connected tree clasper of degree $k$ is called a \emph{$C_k$-tree}. 
A $C_k$-graph \emph{with loop} is a $C_k$-graph which is not a $C_k$-tree.  
\begin{convention}
Throughout this paper, we make use of the following graphical convention. 
The drawing convention for claspers are those of \cite[Fig. 7]{H}, except for the following: a $\oplus$ (resp. $\ominus$) on an edge represents a positive (resp. negative) half-twist. (This replaces the convention of a circled $S$ (resp. $S^{-1}$) used in \cite{H}).    
When representing a clasper $c$ with an edge marked by a $\ast$, we implicitly also define the clasper $c^{-1}$ which is obtained from $c$ by inserting a \emph{positive half twist} in the $\ast$-marked edge.  
Likewise, when introducing the string link $\si$ obtained from $\1$ by surgery along a clasper $c$ with a $\ast$-marked edge, we implicitly also introduce the string link $\si^{-1}$ obtained from $\1$ by surgery along $c^{-1}$.  
(This convention/notation is motivated by Lemma \ref{calculus}(2)).  
We will also make use of this convention for knots in $S^3$.  
\end{convention}

Given a graph clasper $G$ for a string link $\si$, there is a procedure to construct, in a regular neighbourhood of $G$, a framed link $\gamma(G)$.  There is thus a notion of \emph{surgery along $G$}, which is defined as surgery 
along $\gamma(G)$.  
There exists a canonical diffeomorphism between $D^2\times I$ and the manifold $(D^2\times I)_{\gamma(G)}$, and    
surgery along the $C_k$-graph $G$ can be regarded as an operation on $\si$ in the (fixed) ambient space  $D^2\times I$.  
We say that the resulting string link $\si_G$ in $D^2\times I$ is obtained from $\si$ by surgery along $G$.  
In particular, surgery along a simple $C_k$-tree is a local move as illustrated in Figure \ref{ckmove}, which is equivalent to a $C_k$-move as defined in the introduction (Figure \ref{cnm}).  
\begin{figure}[!h]
\includegraphics{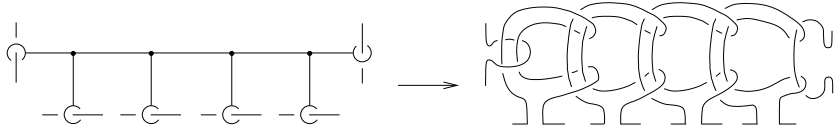}
\caption{Surgery along a simple $C_5$-tree.} \label{ckmove}
\end{figure}

A $C_k$-tree $G$ having the shape of the tree clasper in Figure \ref{ckmove} is called \emph{linear}, and the left-most and right-most leaves of $G$ in Figure \ref{ckmove} are called the \emph{ends} of $G$. 
 
The $C_k$-equivalence (as defined in the introduction) coincides with the equivalence relation on string links generated by surgeries along $C_k$-graphs and isotopies.  
In particular, it is known that two links are $C_k$-equivalent if and only if they are related by surgery along simple $C_k$-trees \cite[Thm. 3.17]{H}. 
\subsection{Calculus of Claspers}
In this subsection, we summarize several properties of the theory of clasper, whose proofs can be found in \cite{H}. 
\begin{lemma}[Calculus of Claspers] \label{calculus}
(1). Let $T$ be a union of $C_k$-trees for a string link $\si$, and let $T'$ be obtained from $T$ by passing an edge across $\si$ or across another edge of $T$, or by sliding a leaf over a leaf of another component of 
$T$.\footnote{For example, the clasper $G_U$ of Figure \ref{asihxstu_fig} is obtained from $G_T$ by sliding a leaf over another one.  } Then $\si_T \stackrel{C_{k+1}}{\sim} \si_{T'}$.  

(2). Let $T$ be a $C_k$-tree for $\1_n$ and let $\ov{T}$ be a $C_k$-tree obtained from $T$ by adding a half-twist on an edge. Then 
 $(\1_n)_T\cdot (\1_n)_{\ov{T}} \stackrel{C_{k+1}}{\sim} \1_n$.  

(3). Let $T$ be a $C_k$-tree for $\1_n$. Let $f_1$ and $f_2$ be two disks obtained by splitting a leaf $f$ of $T$ 
along an arc $\alpha$ as shown in Figure \ref{split}.
 \begin{figure}[!h]
  \includegraphics{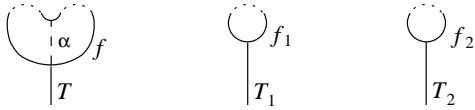}
  \caption{The $3$ claspers are identical outside a small ball, where they are as depicted. }\label{split}
 \end{figure}
Then, $(\1_n)_T \stackrel{C_{k+1}}{\sim} (\1_n)_{T_1}\cdot (\1_n)_{T_2}$, where $T_i$ denotes the $C_k$-tree for $\1_n$ obtained from $T$ by replacing $f$ by $f_i$ ($i=1,2$), see Figure \ref{split}.  
\end{lemma}
In our proofs, we shall use combinations of these relations in many places, and will always refer to them as \emph{Calculus of Claspers}.  

Claspers also satisfy relations analogous to the AS, IHX and STU relations for Jacobi diagrams \cite{BNv}.  
\begin{lemma}\label{asihxstu}

(AS). Let $T$ and $T'$ be two $C_k$-graphs for $\1_n$ which differ only in a small ball as depicted in Figure \ref{asihxstu_fig}.  Then $(\1_n)_{T}\cdot (\1_n)_{T'} \stackrel{C_{k+1}}{\sim}\1_n$. 
 \begin{figure}[!h]
  \includegraphics{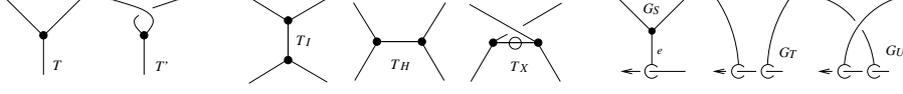}
  \caption{The AS, IHX and STU relations.  }\label{asihxstu_fig}
 \end{figure}

(IHX). Let $T_I$, $T_H$ and $T_X$ be three $C_k$-graphs for $\1_n$ which differ only in a small ball as depicted in Figure \ref{asihxstu_fig}.  Then $(\1_n)_{T_I} \stackrel{C_{k+1}}{\sim} (\1_n)_{T_H}\cdot (\1_n)_{T_X}$.  

(STU). Let $G_S$, $G_T$ and $G_U$ be three 
$C_k$-graphs for $\1_n$ which differ only in a small ball as depicted in Figure \ref{asihxstu_fig}.  Then 
$(\1_n)_{G_S}\cdot (\1_n)_{G_T} \stackrel{C_{k+1}}{\sim} (\1_n)_{G_U}$. 
\end{lemma}

In the rest of the paper, we will simply refer to Lemma \ref{asihxstu} as the AS, IHX and STU relations.  In some cases, it will be convenient to also use the following terminology.  If $e$ denote the edge of a graph clasper $G_S$ (resp. if $f$ and $f'$ denote the leaves of $G_T$ or $G_U$) as in Figure \ref{asihxstu_fig}, we will sometimes say that we apply the STU relation \emph{at the edge $e$} (resp. at the leaves $f$ and $f'$) when applying Lemma \ref{asihxstu}(STU).  

 Note that the STU relation stated above differs by a sign from the STU relation for Jacobi diagrams.  
 Note also that, in contrast to the Jacobi diagram case, it only holds among \emph{connected} claspers.  

We conclude this subsection with an additional lemma which will be used later.  We first need a couple of extra notation.  

Let $k>2$ and $l\in \ks$ be integers.  Denote by $\mathcal{B}_k(l)$ the set of all bijections $\alpha: \{1,...,k-1\}\longrightarrow\{1,...,k\}\setminus\{l\}$ such that $\alpha(1)<\alpha(k-1)$.   
We denote by $id\in \mathcal{B}_k(l)$ the element which maps $i$ to itself if $1\le i<l$, and to $i+1$ otherwise.  
For each $\alpha\in \mathcal{B}_k(l)$, let $T_{\alpha}(l)$ and $\overline{T_{\alpha}}(l)$ denote the $C_k$-trees for $\1_k$ represented in Figure \ref{flipping_fig}.  Denote respectively by $B_{\alpha}(l)$ and $\overline{B_{\alpha}}(l)$ the $k$-string links obtained from $\1_n$ by surgery along $T_{\alpha}(l)$ and $\overline{T_{\alpha}}(l)$.  
 \begin{figure}[!h]
  \includegraphics{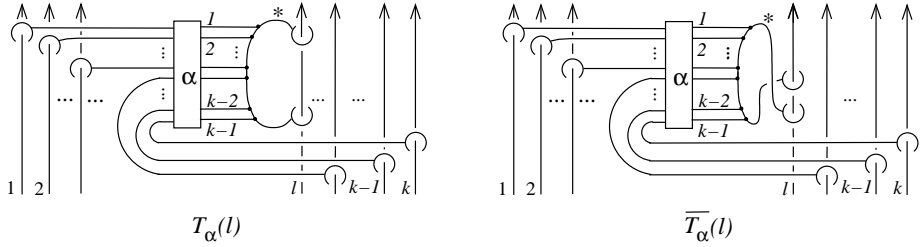}
  \caption{The $C_k$-trees $T_{\alpha}(l)$ and $\overline{T_{\alpha}}(l)$. }\label{flipping_fig}
 \end{figure}

\begin{lemma}\label{flipping}
Let $k>2$ and $l\in \ks$ be integers.  For any $\alpha\in \mathcal{B}_k(l)$ and any integer $l'~(1\leq l'\leq k,~l'\neq l)$, there is a bijection $\beta\in \mathcal{B}_k(l')$ such that $B_{\alpha}(l)\cdot \left(\overline{B_{\alpha}}(l)\right)^{-1}\stackrel{C_{k+1}}{\sim} B_{\beta}(l')\cdot \left(\overline{B_{\beta}}(l')\right)^{-1}$.
\end{lemma}
\begin{proof}
Observe that $T_{\alpha}(l)$ and $\overline{T_{\alpha}}(l)$ are identical except in a $3$-ball where they look exactly like $G_T$ and $G_U$ in Figure \ref{asihxstu_fig}.  
So by the STU relation we have $B_{\alpha}(l)\cdot \left(\overline{B_{\alpha}}(l)\right)^{-1}\stackrel{C_{k+1}}{\sim} \1_G$, where $G$ is a $C_k$-graph intersecting each component of $\1_k$ once.  Note that $G$ has one loop, and that each leaf of $G$ is connected to the loop by a single edge.   
So for each $1\leq l'\leq k$ we can apply the STU relation at the edge of $G$ which is attached to the leaf intersecting the $l'$th component of $\1_k$.  This gives the desired formula.  
\end{proof}
\subsection{$k$-additivity} \label{add}
We now introduce the notion of $k$-additivity of a string link invariant.  
\begin{definition} \label{def_kadd}
Let $k, n\ge 1$ be integers.  We say that an invariant $v: \mathcal{SL}(n)\rightarrow \mathbb{Z}$ is \emph{$k$-additive} 
if for every $\si\in \mathcal{SL}(n)$ and every $\si'\in \mathcal{SL}_k(n)$, we have $v(\si\cdot \si')=v(\si)+v(\si')$.  
\end{definition}

Note that a string link invariant is additive if and only if it is $1$-additive.  Note also that for $k>l$, 
the $l$-additivity implies the $k$-additivity.  
We now show that all the invariants involved in our classification results are $k$-additive for some $k$. 

First, Milnor invariants $\mu(I)$ of length $|I|=k$ are $(k-1)$-additive.  This follows from Milnor invariants' additivity property \cite[Lem. 3.3]{MY} and the fact that Milnor invariants of length $k$ are $C_k$-equivalence invariants \cite[Thm. 7.2]{H}. 

Now, observe that the plat closure of the product of two $n$-string links $\si\cdot \si'$ is just the connected sum of their plat closures.  So it follows, by the multiplicativity of the Conway and HOMFLYPT polynomial (see Subsection \ref{knotfti}), that $f_2$ is $2$-additive, $f_3$ is $3$-additive, and $f^1_4$ and $f^2_4$ are both $4$-additive.  

Next we prove the following.  
\begin{claim}\label{claimadditivity}
Let $\si\in \mathcal{SL}(3)$, and $\si'\in \mathcal{SL}_k(3)$ for an integer $k\ge 1$.  
Then for each $i=0,...,4$, the closure $cl_i$ of $\si\cdot \si'$ satisfies $cl_i (\si\cdot \si')\stackrel{C_{k+1}}{\sim} (cl_i\si)\sharp (cl_i\si')$.
\end{claim}
\begin{proof}[Proof of Claim \ref{claimadditivity}]
By \cite[Thm. 3.17]{H}, we have $\si'=(\1_3)_G$, where $G$ is a disjoint union of simple $C_k$-trees for $\1_3$.  Let $i\in\{0,...,4\}$.  Using Calculus of Claspers, we have 
$cl_i(\si\cdot \si')=cl_i(\si\cdot (\1_3)_G )\stackrel{C_{k+1}}{\sim} cl_i(\si\cdot (\1_3)_{G'} )$, 
where $G'$ is a union of $C_k$-trees for $\1_3$ which is contained in a tubular neighbourhood of the first strand.  
Clearly, we have $cl_i(\si\cdot (\1_3)_{G'} )=(cl_i\si)\sharp (cl_i(\1_3)_{G'} )$.  
On the other hand, it can be easily checked that $cl_i\si'=cl_i(\1_3)_{G}\stackrel{C_{k+1}}{\sim} cl_i(\1_3)_{G'}$.  This concludes the proof.  
\end{proof}
It follows from Claim \ref{claimadditivity} and the multiplicativity of the Conway and HOMFLYPT polynomial that $V_3$ is $3$-additive and that $V_4^i$ is a $4$-additive invariant for $i=1,...,7$.  Similar arguments on the closures $K_i$ ($i=1,2,3$) show that each invariant $W_4^i$, $i=1,2,3$ is also $4$-additive.  

Finally, we can use Lemma \ref{calculus}(3) to show the following.  
\begin{claim} \label{claimadd2}
Let $\si\in \mathcal{SL}(2)$, and $\si'\in \mathcal{SL}_k(2)$ for some integer $k\ge 1$.  
Then for $i=1,2$ and for $j=0,...,4$, we have $cl_j(\Delta_i (\si\cdot \si'))\stackrel{C_{k+1}}{\sim} cl_j(\Delta_i\si)\sharp cl_j(\Delta_i\si')$.
\end{claim}
\begin{proof}[Proof of Claim \ref{claimadd2}]
As in the previous proof, we have $\si'=(\1_2)_G$ for a disjoint union $G$ of simple $C_k$-trees for $\1_2$.  
For simplicity, we give here the proof on a simple example, namely in the case where $i=1$ and where $G$ is (say) a copy of the $C_4$-tree $s$ represented on the left-hand side of Figure \ref{proofc4fig}.  (The general case is proved by strictly similar arguments).  
 \begin{figure}[!h]
  \includegraphics{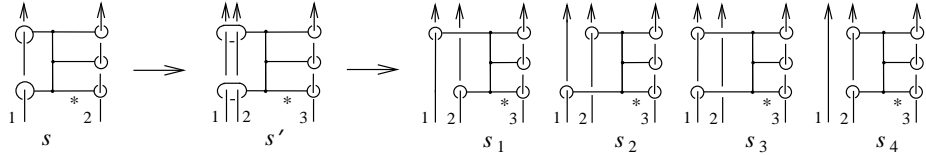}
  \caption{Doubling the first component of $\1_2$ and splitting the leaves. }\label{proofc4fig}
 \end{figure} 
Starting with $s$, doubling the first component of $\1_2$ yields a $C_4$-tree $s'$ for $\1_3$ as illustrated in Figure \ref{proofc4fig}.  We apply Lemma \ref{calculus}(3) repeatedly, to decompose $s'$ into simple $C_4$-trees.  This gives $(\1_3)_{s'}\stackrel{C_5}{\sim} \prod_{1\le i\le 4} (\1_3)_{s_i}$, 
where $s_i$ is a simple $C_4$-tree for $\1_3$ as illustrated in Figure \ref{proofc4fig}, $i=1,2,3,4$. 
The result then follows by Claim \ref{claimadditivity}.
\end{proof}
This claim implies that $f^3_4$, $f^4_4$ and $f^5_4$ are also $4$-additive invariants.   
\subsection{The clasper index}
Let $G$ be a simple $C_k$-graph for an $n$-string link $\si$.  
We call a leaf of $G$ an $i$-leaf if it intersects the $i$th component of $\si$.  
The \emph{index} of $G$ is the collection of all integers $i$ such that $G$ contains an $i$-leaf, counted with multiplicities. For example, a simple $C_3$-tree of index $\{2,3^{(2)},5\}$ for $\si$ intersects twice component $3$ and once components $2$ and $5$ (and is disjoint from all other components of $\si$).  

We will need the following lemma.  
\begin{lemma}\label{index1}
For $k\ge 3$, let $T$ be a simple $C_k$-tree of index $\{i,j^{(k)}\}$ for an $n$-string link $\si$, $1\le i,j\le n$. 
Then $\si_T$ is $C_{k+1}$-equivalent to a string link $\si'$ which is obtained from $\si$ by surgery along $C_k$-trees with index $\{i^{(2)},j^{(k-1)}\}$. 
\end{lemma}

\begin{proof} 
For simplicity we prove the lemma for $\si=\1_n$.  For an arbitrary $\si$, the proof is strictly similar (using the fact that there exists a tree clasper $C$ such that $\si=(\1_n)_C$).  
Pick a node of $T$ which is connected to two $j$-leaves $f$ and $f'$.  Travelling along the $j$th component of $\1_n$ from $f$ to $f'$, we meet in order $m$ $j$-leaves $f_1$, ..., $f_m$.  
The proof is by induction on the number $m$ of leaves separating $f$ and $f'$.  

If $m=0$, then using Calculus of Claspers we may assume that there exists a $3$-ball which intersects $T$ as on the left-hand side of Figure \ref{fork}.  
By the IHX and STU relations, we have $\1_T\stackrel{C_{k+1}}{\sim} \1_G$, where $G$ is a simple $C_k$-graph with one loop and with index $\{i,j^{(k-1)}\}$ as illustrated in Figure \ref{fork}.  
 \begin{figure}[!h]
  \includegraphics{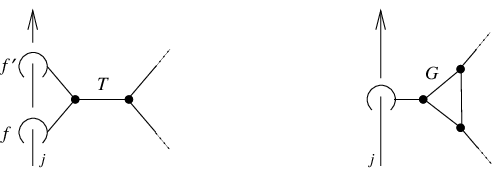}
  \caption{}\label{fork}
 \end{figure}
 We now prove that any simple $C_k$-graph $C$ for $\1_n$ with one loop and with index $\{i,j^{(k-1)}\}$ satisfies 
 \begin{equation}\label{claimloopequ}
 (\1_n)_C\stackrel{C_{k+1}}{\sim} (\1_n)_F,
 \end{equation} 
 where $F$ is a disjoint union of simple $C_k$-trees for $\1_n$ with index $\{i^{(2)},j^{(k-1)}\}$.   
%
In order to prove (\ref{claimloopequ}), observe that the unique $i$-leaf $l$ of $C$ is connected to the loop $\gamma$ of $C$ by a path $P$ of edges and nodes.  
We proceed by induction on the number $n$ of nodes in $P$. 
For $n=0$, applying the STU relation at the edge connecting $l$ to $\gamma$ proves the claim.  
For an arbitrary $n\ge 1$, applying the IHX relation at the edge of $P$ which is incident to $\gamma$ gives $\1_{C}\stackrel{C_{k+1}}{\sim} \1_{C'}\cdot \1_{C''}$, where $C'$ and $C''$ are two simple $C_k$-graphs 
with a unique $i$-leaf connected to a loop by a path with $(n-1)$ nodes.  Equation (\ref{claimloopequ}) then follows from the induction hypothesis.  

Now suppose that $f$ and $f'$ are separated by $m$ $j$-leaves $f_1$, ... , $f_m$ ($m\ge 1$).  We can apply the STU relation at the leaves $f_m$ and $f'$ to obtain that $\1_T\stackrel{C_{k+1}}{\sim} \1_{T'}\cdot \1_{G}$, 
where $T'$ is the $C_k$-tree obtained by sliding $f_m$ over $f'$ (so that the $j$-leaves $f$ and $f'$ of $T'$ are separated by $m-1$ leaves)\footnote{Abusing notation, we still call $f$ and $f'$ the corresponding $j$-leaves of $T'$.  }, and where $G$ is a simple $C_k$-graph with one loop and with index $\{i,j^{(k-1)}\}$.  
The result thus follows from (\ref{claimloopequ}) and the induction hypothesis.  
\end{proof}
%
%
%
%
\section{Proofs of the main results}
In this section we give the proofs of Theorems \ref{c3}, \ref{c4} and \ref{c5}.  
The plan of proof is always the same and as follows.  That $(1)\Rightarrow (2)\Rightarrow (4)$ is clear, so the core of the proof consists in showing that $(4)\Rightarrow (1)$.  This is done by giving an explicit representative for the $C_k$-equivalence class ($k=3,4,5$) of an arbitrary $n$-string link, in terms of the invariants listed in $(4)$.  
That $(3)\Leftrightarrow (2)$ follows from the fact that the group $\mathcal{SL}(n) / C_k$ is torsion-free for $k=3,4,5$, which comes as a consequence of the fact that no torsion element appears in our representative.  

Before proceeding to the proofs, we summarize in Figure \ref{knots}, for the reader's convenience, the various knots that will be used throughout the rest of this section.  (We implicitly define the knots in Figure \ref{knots} as the results of surgery along the represented tree claspers for the unknot $U$).  We will sometimes identify these knots with their images by the monoid isomorphism $\mathcal{L}(1)\simeq \mathcal{SL}(1)$.
Also, for each knot $K$ in Figure \ref{knots} and for any $1\le i\le n$, we will denote by $K_i$ the $n$-string link obtained from $\1_n$ by connected sum of a copy of $K$ on the $i$th component.  
 \begin{figure}[!h]
  \includegraphics{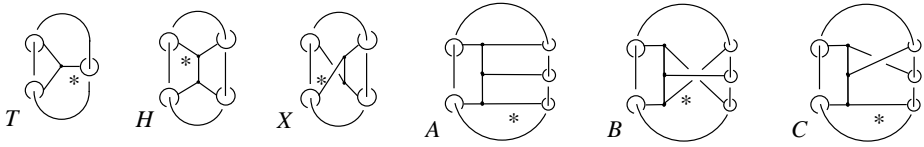}
  \caption{The knots $T$, $H$, $X$, $A$, $B$ and $C$. } \label{knots}
 \end{figure} 
\subsection{Proof of Theorem \ref{c3}} \label{proofc3}
 Let $\si\in \mathcal{SL}(n)$.   
 By Murakami-Nakanishi's characterization of $C_2$-equivalence \cite{MN}, we have that $\si$ is $C_2$-equivalent to 
 \begin{equation}\label{L1}
  \si_{(1)}:=\prod_{1\le i<j\le n} L_{ij}^{\mu_\si(ij)}, 
 \end{equation}
 where $L_{ij}\in \mathcal{SL}(n)$ is obtained by surgery along the $C_1$-tree $l_{ij}$ represented in Figure \ref{c2tree}.  
 So $\sigma$ is obtained from $\si_{(1)}$ by surgery along $C_k$-trees ($k\ge 2$).  By Calculus of Claspers, this implies that $\si\stackrel{C_3}{\sim} \si_{(1)}\cdot \si_{(2)}$, where 
  \begin{equation}\label{L2}
   \si_{(2)}:=
       \prod_{1\le i\le n} T_i^{\alpha_i}\cdot 
       \prod_{1\le i<j\le n} W_{ij}^{\beta_{ij}}\cdot 
       \prod_{1\le i<j<k\le n} B_{ijk}^{\gamma_{ijk}}, 
  \end{equation}
 for some integers $\alpha_i$, $\beta_{ij}$ and $\gamma_{ijk}$, where $T_i$, $W_{ij}$ and $B_{ijk}$ are $n$-string links obtained respectively from $\1_n$ by surgery along the $C_2$-trees $t_i$, $w_{ij}$ and $b_{ijk}$ represented in Figure \ref{c2tree}. 
 \begin{figure}[!h]
  \includegraphics{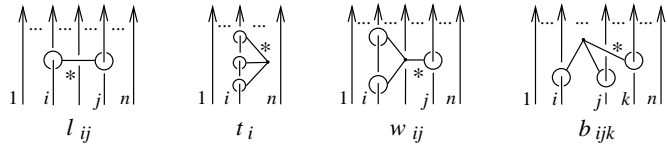}
  \caption{The $C_1$-tree $l_{ij}$ and the $C_2$-trees $t_i$, $w_{ij}$ and $b_{ijk}$. } \label{c2tree}
 \end{figure} 
 Note that the closure of $T_i$, $W_{ij}$ and $B_{ijk}$ is the trefoil, Whitehead link and Borromean rings respectively.  Note also that $W_{ij}=W_{ji}$ (see for example \cite[Fig. 6]{jbjktr}).  
 For $\si\in \mathcal{SL}(n)$, set $(a_2)_i(\si)=a_2(\si_i)$ ($1\le i\le n$) and $(f_2)_{ij}(\si)=f_2(\si_i\cup \si_j)$ ($1\le i,j\le n$).  We have $(a_2)_i(T_j)=\delta_{i,j}$, $(f_2)_{ij}(T_k)=\delta_{i,k}+\delta_{j,k}$, $(f_2)_{ij}(W_{kl})=\delta_{(i,j),(k,l)}$ and $\mu_{B_{abc}}(ijk)=\delta_{(i,j,k),(a,b,c)}$, where $\delta$ denotes the Kronecker delta.  Using the fact that $a_2$, $f_2$ and $\mu(ijk)$ are all $2$-additive, it follows that in (\ref{L2}) we have 
 \bc
 $\alpha_i=(a_2)_i(\si_{(2)})=(a_2)_i(\si)$,\\
 $\beta_{ij}=(f_2)_{ij}(\si_{(2)})-(a_2)_j(\si_{(2)})-(a_2)_i(\si_{(2)})=(f_2)_{ij}(\si)-(a_2)_j(\si)-(a_2)_i(\si)$, \\
 and $\textrm{ }\gamma_{ijk}=\mu_{\si_{(2)}}(ijk)=\mu_{\si}(ijk)$.  
 \ec
 This concludes the proof.  
\subsection{Proof of Theorem \ref{c4}} \label{proofc4}
 Let $\si\in \mathcal{SL}(n)$.   
 From the proof of Theorem \ref{c3}, we have that $\si\stackrel{C_3}{\sim} \si_{(1)}\cdot \si_{(2)}$ where $\si_{(1)}$ and $\si_{(2)}$ are given by (\ref{L1}) and (\ref{L2}) respectively, and the exponents $\alpha_i$, $\beta_{ij}$ and $\gamma_{ijk}$ in (\ref{L2}) are uniquely determined by the invariants $a_2$, $f_2$ and $\mu(ijk)$ of $\si$.  

 It follows, by Calculus of Claspers, that $\si\stackrel{C_4}{\sim} \si_{(1)}\cdot \si_{(2)}\cdot \si_{(3)}$ with 
 $$ \si_{(3)}:=(\1_n)_{G_1}\cdot ( \1_n)_{G_2} \cdot  ... \cdot ( \1_n)_{G_N}, $$
 where, for each $k$, $G_k$ is a simple $C_3$-tree for $\1_n$.  
 Set $G=\sqcup_k G_k$.    
 By Lemma \ref{index1}, we may assume that each $G_k$ in $G$ has index $\{i^{(4)}\}$, $\{i^{(2)},j^{(2)}\}$, $\{i,j,k^{(2)}\}$ or $\{i,j,k,l\}$, for some indices $i,j,k,l\in \n$.  
 Let us consider each of these four cases successively.  \\
 
 \textit{Index $\{i^{(4)}\}$}:  
 Let $F_i\subset G$ denote the union of $C_3$-trees with index $\{i^{(4)}\}$, for each $i$.  By Calculus of Claspers we may assume that $F_i$ lives in a tubular neighbourhood of the $i$th strand of $\1_n$.  Let $H_i$ denote the $n$-string link obtained from $\1_n$ by surgery along the  $C_3$-tree $h_i$ represented in Figure \ref{c3tree}.  The knot obtained by closing the $i$th strand of $H_i$ is the knot $H$ of Figure \ref{knots}.  By \cite{horiuchi} we know that $\mathbf{h}:=P_0^{(3)}(H)$ is nonzero.  It thus follows from Theorem \ref{cnknots}, and the fact that $P_0^{(3)}$ is the only degree $3$ knot invariant, that $(\1_n)_{F_i}\stackrel{C_4}{\sim} (H_i)^{P_0^{(3)}(\si_{(3)}) / \mathbf{h}}$.  
 \begin{figure}[!h]
  \includegraphics{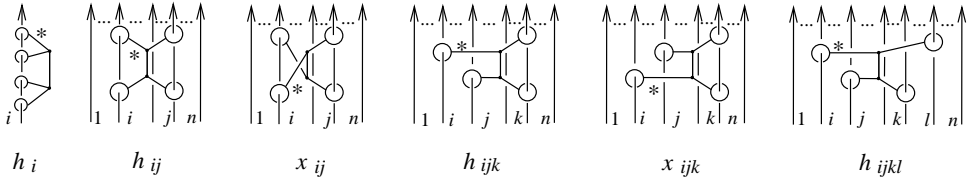}
  \caption{The $C_3$-trees $h_i$, $h_{ij}$, $x_{ij}$, $h_{ijk}$, $x_{ijk}$ and $h_{ijkl}$.  } \label{c3tree}
 \end{figure} 
 
 \textit{Index $\{i^{(2)},j^{(2)}\}$}:
 Fix $i<j\in \n$, and let $F_{ij}\subset G$ denote the union of $C_3$-trees with index $\{i^{(2)},j^{(2)}\}$.  By the AS and IHX relations we may assume that both ends of each $C_3$-tree in $F_{ij}$ are $j$-leaves.  Hence we have 
  \begin{equation} \label{eqc4_1}
   (\1_n)_{F_{ij}}\stackrel{C_4}{\sim} H_{ij}^{a_{ij}}\cdot X_{ij}^{b_{ij}}
  \end{equation}
for some integers $a_{ij}$ and $b_{ij}$, where $H_{ij}$ and $X_{ij}$ denote the $n$-string links obtained respectively from $\1_n$ by surgery along the $C_3$-trees $h_{ij}$ and $x_{ij}$ shown in Figure \ref{c3tree}.  

 A direct computation shows that $\mu_{H_{ij}}(iijj)=\mu_{X_{ij}}(iijj)=2$, and clearly we have $\mu_{(\1_n)_{G_p}}(iijj)=0$ for any $G_p\subset G$ with index $\ne \{i^{(2)},j^{(2)}\}$.  Now, for $\s\in \mathcal{SL}(n)$ and $1\le k<l\le n$, set $(f_3(\si))_{k,l}:=f_3(\si_k\cup \si_l)=P_0^{(3)}(\overline{\si_k\cup \si_l})$.  By \cite{horiuchi}, we have 
 \[ \textrm{$(f_3(H_{ij}))_{k,l}=\delta{(i,j),(k,l)}\cdot \mathbf{h}$ and $(f_3(X_{ij}))_{k,l}=0$.} \]  
 Also, we have $(f_3(H_i))_{k,l}=(\delta_{i,k} + \delta_{i,l})\cdot \mathbf{h}$, and $(f_3((\1_n)_{G_p}))_{k,l}=0$ for any $G_p\subset G$ with index other than $\{k^{(4)}\}$, $\{l^{(4)}\}$ or $\{k^{(2)},l^{(2)}\}$.  
 It follows that the integers $a_{ij}$ and $b_{ij}$ in (\ref{eqc4_1}) are uniquely determined by the invariants $P^{(3)}_0$,  $f_3$ and $\mu(iijj)$ of $\si_{(3)}$.  \\

 \textit{Index $\{i,j,k^{(2)}\}$}:    
  Fix $i,j,k\in \n$ with $i<j$, and let $F_{ijk}\subset G$ denote the union of $C_3$-trees with index $\{i,j,k^{(2)}\}$.  By the AS and IHX relation we may assume that 
  $$ (\1_n)_{F_{ijk}}\stackrel{C_4}{\sim} H_{ijk}^{\alpha_{ijk}}\cdot X_{ijk}^{\beta_{ijk}} $$
 for some integers $\alpha_{ijk}$ and $\beta_{ijk}$, where $H_{ijk}$ and $X_{ijk}$ denote the $n$-string links obtained respectively from $\1_n$ by surgery along the $C_3$-trees $h_{ijk}$ and $x_{ijk}$ represented in Figure \ref{c3tree}.  
  Note that $H_{ijk}$ and $X_{ijk}$ correspond to the string links $V_{id}(3)$ and $\overline{V_{id}}(3)$ defined for Lemma \ref{flipping} respectively (using Lemma \ref{calculus}~(2) for the second one).  Thus by Lemma \ref{flipping}, the union $F_{(3)}:=\cup_{i,j,k} F_{ijk}$ of all $C_3$-trees in $G$  intersecting $3$ strands of $\1_n$ satisfies 
  \begin{equation} \label{eqc4_2}
   (\1_n)_{F_{(3)}}\stackrel{C_4}{\sim} \prod_{1\le i<j<k\le n} (H_{jki})^{a_{ijk}}\cdot (H_{ikj})^{b_{ijk}}\cdot (H_{ijk})^{c_{ijk}}\cdot (X_{ijk})^{d_{ijk}}
  \end{equation}
 for some integers $a_{ijk}$, $b_{ijk}$, $c_{ijk}$ and $d_{ijk}$.  
 
 We have $\mu_{H_{ijk}}(ijkk)=\mu_{X_{ijk}}(ijkk)=1$ for all $1\le i,j,k\le n$ with $i<j$, and $\mu_{(\1_n)_{G_p}}(ijkk)=0$ for any $G_p\subset G$ with index $\ne \{i,j,k^{(2)}\}$.   

 For $\s\in \mathcal{SL}(n)$ and $1\le i<j<k\le n$, set $cl^{ijk}_3(\si):=cl_3(\si_i\cup \si_j\cup \si_k)$ and $(V_3)_{ijk}(\si):=V_3(cl^{ijk}_3(\si))$.  
 We have 
 \[ \textrm{$cl^{ijk}_3(X_{ijk})\stackrel{C_5}{\sim} H$ and $cl^{ijk}_3(H_{ikj})\stackrel{C_5}{\sim} cl^{ijk}_3(H_{jki})\stackrel{C_5}{\sim} cl^{ijk}_3(H_{ijk})\stackrel{C_5}{\sim} X$,} \]
  where $H$ and $X$ are the two knots represented in Figure \ref{knots}.  Note that by the IHX relation we have $X\stackrel{C_4}{\sim} U$.  More generally, we compute the closures $cl^{ijk}_3$ and invariants $(V_3)_{ijk}$ of the relevant $n$-string links in Table \ref{t0}.  
\begin{table}[!h]
\begin{center}
\begin{tabular}{|c|c|c|}
\hline
 $\si$ & $cl^{ijk}_3(\si) / C_4$ & $(V_3)_{ijk}(\si) / \mathbf{h}$ \rule[0pt]{-5pt}{10pt} \\[0.1cm] 
\hline
$H_a$ & $H^{\delta_{a,i}}\cdot H^{\delta_{a,j}}\cdot H^{\delta_{a,k}}$ & $\delta_{a,i}+\delta_{a,j}+\delta_{a,k}$ \rule[0pt]{-5pt}{10pt} \\[0.1cm] 
\hline	 
$H_{ab}$ ($a<b$) & $H^{\delta_{(a,b),(i,j)}}\cdot H^{\delta_{(a,b),(j,k)}}\cdot X^{\delta_{(a,b),(i,k)}}$ & $\delta_{(a,b),(i,j)}+\delta_{(a,b),(j,k)}$ \rule[0pt]{-5pt}{10pt} \\[0.1cm] 
\hline	
$X_{ab}$ ($a<b$) & $H^{\delta_{(a,b),(i,k)}}\cdot X^{\delta_{(a,b),(i,j)}}\cdot X^{\delta_{(a,b),(j,k)}}$ & $\delta_{(a,b),(i,k)}$ \rule[0pt]{-5pt}{10pt} \\[0.1cm] 
\hline	
$H_{abc}$ ($a<b<c$) & $X^{\delta_{(a,b,c),(i,j,k)}}$ & $0$ \rule[0pt]{-5pt}{10pt} \\[0.1cm]
\hline	
$X_{abc}$ ($a<b$) & $H^{\delta_{(a,b,c),(i,j,k)}}$ & $\delta_{(a,b,c),(i,j,k)}$ \rule[0pt]{-5pt}{10pt} \\[0.1cm]  
\hline	
\end{tabular}\\[0.1cm]
\caption{ } \label{t0}
\end{center}
\end{table} 

It follows that all exponents in (\ref{eqc4_2}) are uniquely determined by the invariants $P^{(3)}_0$, $f_3$, $V_3$ and $\mu(ijkk)$ ($1\le i, j,k\le n$ ; $i<j$) of $\si_{(3)}$. \\
 
 \textit{Index $\{i,j,k,l\}$}: 
 By the IHX and AS relations, we may assume that the $k$-leaf and $l$-leaf of any $C_3$-tree $C\subset G$ with index $\{i,j,k,l\}$ ($i<j<k<l$) are its two ends.  
More precisely, the union $F_{(4)}\subset G$ of all $C_3$-trees intersecting $4$ distinct components of $\1_n$ satisfies 
$$ (\1_n)_{F_{(4)}}\stackrel{C_4}{\sim} \prod_{\substack{1\le i<j<k<l\le n}} 
(H_{ijkl})^{\mu_{\si_{(3)}}(ijkl)}\cdot (H_{jikl})^{\mu_{\si_{(3)}}(jikl)}, $$
where $H_{ijkl}$ denotes the $n$-string link obtained from $\1_n$ by surgery along the $C_3$-trees $h_{ijkl}$ represented in Figure \ref{c3tree}.  This follows from the fact that $\mu_{H_{ijkl}}(i'j'k'l')=\delta_{(i,j,k,l),(i',j',k',l')}$ \cite{Milnor,HM}.  
So we have shown that $\si_{(3)}$ is $C_4$-equivalent to 
 \begin{equation}\label{L3}
 \prod_{i} (H_i)^{a_i}\cdot 
 \prod_{i<j} (H_{ij})^{a_{ij}}\cdot (X_{ij})^{b_{ij}}\cdot 
 \prod_{\substack{i<j \\ k}} (H_{ijk})^{a_{ijk}}\cdot 
 \prod_{i<j<k} (X_{ijk})^{b_{ijk}}\cdot 
 \prod_{\substack{i,j \\ k<l}} (H_{ijkl})^{a_{ijkl}},
 \end{equation}
where the exponents are integers determined uniquely by the invariants $P_0^{(3)}$, $f_3$ and $V_3$, and Milnor invariants $\mu(iijj)$ ($1\le i<j\le n$), $\mu(ijkk)$ ($1\le i,j,k\le n$ ; $i<j$) and $\mu(ijkl)$ ($1\le i,j< k< l\le n$) of $\si_{(3)}$.   
The result follows from the fact that all the above-listed invariants are $3$-additive.  
\begin{remark}
It appears from the proof of Theorem \ref{c4} (case of index $\{i^{(2)},j^{(2)}\}$ trees) that we can replace, in the statement, `all' invariants $V_3$ of $\sigma$ and $\sigma'$ by (only) the invariants $(V_3)_{ijk}$ for $1\le i<j<k\le n$.  
Indeed, only those, among all invariants $(V_3)_{ijk}$, are used to determine the value of the exponents in (\ref{eqc4_2}).  
\end{remark}
\subsection{Proof of Theorem \ref{c5}} \label{proofc5}
Before proving Theorem \ref{c5} we investigate individually the case of $n$-string links for $n=1$, $2$, $3$ and $4$. 
We start by reviewing briefly the case $n=1$, that is, the knot case.  
\subsubsection{The knot case}\label{subknot}
It is well known that there exists essentially two linearly independent finite type knots invariants of degree $4$, namely $a_4$ and $P_0^{(4)}$.  

For an element $\alpha$ of the symmetric group $S_3$, denote by $K_\alpha$ the knot obtained from the unknot $U$ by surgery along the $C_4$-tree $k_\alpha$ represented in Figure \ref{figc4_1}.  
Note that $K_{id}$, $K_{(13)}$ and $K_{(12)}$ are the three knots $A$, $B$ and $C$ illustrated in Figure \ref{knots}.\footnote{Here, and in the rest of the paper, we denote by $id$ the identity element of the symmetric group.  }    
By the AS and IHX relations, the abelian group $\mathcal{SL}_4(1) / C_5$ is generated by these sixelements $K_\alpha$, $\alpha\in S_3$.  
Further, by using the IHX and STU relations we observe that 
\begin{equation}\label{knotc4}
K_{(12)}\stackrel{C_5}{\sim} K_{(23)}\stackrel{C_5}{\sim} U\quad \textrm{ and }\quad K_{(13)}\stackrel{C_5}{\sim} K_{(123)}\stackrel{C_5}{\sim} K_{(132)}.  
\end{equation}
(In particular, we have that the knot $C$ of Figure \ref{knots} satisfies $C\stackrel{C_5}{\sim} U$.)  
This shows that $\mathcal{SL}_4(1) / C_5$ is generated by the two knots $A=K_{id}$ and $B=K_{(13)}$ of Figure \ref{knots}.   
By using \cite{OY} and \cite{horiuchi}, we have that 
\beQ
a_4{(A)} = 0 & \textrm{ and } & a_4{(B)} = \pm 2, \\
P^{(4)}_0{(A)} = \pm 4!.2^4 & \textrm{ and } & P^{(4)}_0{(B)} = 0.  
\eeQ
Set $\mathbf{a}:=P^{(4)}_0{(A)}$ and $\mathbf{b}:=a_4{(B)}$.  The $C_5$-equivalence class of a knot $K$ is thus determined by its degree $\le 4$ invariants as follows 
$$ K\stackrel{C_5}{\sim} T^{a_2(K)}\cdot H^{P^{(3)}_0(K) / \mathbf{h}}\cdot A^{P^{(4)}_0(K) / \mathbf{a}}\cdot B^{a_4(K) / \mathbf{b}},$$
where $T$ and $H$ are given in Figure \ref{knots}.  
\subsubsection{The $2$-component case}\label{sub2comp}
We aim to prove the following particular case of Theorem \ref{c5}. 
\begin{lemma}\label{claimc4_2}
Let $\si$, $\si'\in \mathcal{SL}_4(2)$. Then $\si$ and $\si'$ are $C_{5}$-equivalent if and only if they share all knots invariants of degree $4$ and the five invariants $f^i_4$ ($i=1,...,5$).  
\end{lemma}
\noindent Note that there is no nontrivial Milnor invariant of length $5$ for $2$-string links \cite{Milnor2}.  
\begin{proof}[Proof of Lemma \ref{claimc4_2}]
Let $\si\in \mathcal{SL}_4(2)$.   
By Calculus of Claspers and Subsection \ref{subknot}, 
$$ \si\stackrel{C_5}{\sim} 
\prod_{i=1}^2\left( A^{P^{(4)}_0(K) / \mathbf{a}}\cdot B^{a_4(K) / \mathbf{b}} \right)\cdot (\mathbf{1}_2)_{F} $$  
where $F$ is a disjoint union of simple $C_4$-trees for $\mathbf{1}_2$ with at least one $1$-leaf and one $2$-leaf.  Note that by Lemma \ref{index1}, we may assume that each $C_4$-tree in $F$ has index $\{1^{(2)},2^{(3)}\}$ or $\{1^{(3)},2^{2)}\}$.  
It follows, by the IHX relation, that the abelian group $\mathcal{SL}_4(2) / C_5$ is generated by $A_i$, $B_i$ ($i=1,2$) and the $2$-string links $\si^1_\alpha$ and $\si^2_\alpha$ obtained from $\mathbf{1}_2$ by surgery along the $C_4$-trees $s^1_\alpha$ and $s^2_\alpha$ represented in Figure \ref{figc4_1} ($\alpha\in S_3$).   
 \begin{figure}[!h]
  \includegraphics{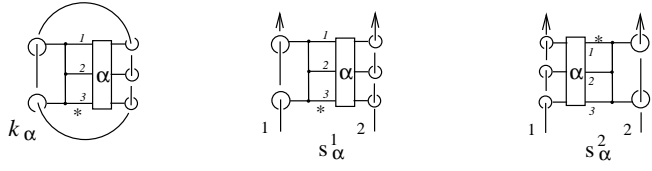}
  \caption{The $C_4$-trees $k_\alpha$, $s^1_\alpha$ and $s^2_\alpha$ ; $\alpha\in S_3$. } \label{figc4_1}
 \end{figure} 

We can use the AS and STU relations to prove the following relations ($i=1,2$): 
\beQ
 \si^i_{(12)}\stackrel{C_5}{\sim} \si^i_{id}\cdot (\1_2)_{g_0} & , & \si^i_{(23)}\stackrel{C_5}{\sim} \si^i_{id}\cdot (\1_2)_{g_1}, \\
 \si^i_{(123)} \stackrel{C_5}{\sim} \si^i_{(23)}\cdot (\1_2)_{g_2} & , & \si^i_{(132)} \stackrel{C_5}{\sim} \si^i_{(12)}\cdot (\1_2)_{g_2}, \\
 \si^i_{(13)} & \stackrel{C_5}{\sim} & \si^i_{(123)}\cdot (\1_2)_{g_3}, 
\eeQ
where $g_k$ ($k=0,1,2,3$) is the $C_4$-graph represented in Figure \ref{c4tr}.  By applying the STU relation at an edge of $g_k$ that connects its loop to a $2$-leaf, we can express $(\1_2)_{g_k}$ as a product of $s^1_\alpha$'s.  So the relations above imply that for any $\alpha\in S_3\setminus \{Id\}$, the string link $s^2_\alpha$ is generated in $\mathcal{SL}_4(2) / C_5$ by $s^2_{id}$ and the $s^1_\alpha$'s.  Further, one can easily check using the IHX relation that $(\1_2)_{g_0}\stackrel{C_5}{\sim} (\1_2)_{g_1}$.  This implies that $\si^1_{(12)}\stackrel{C_5}{\sim} \si^1_{(23)}$, and thus (by the above relations) that $\si^1_{(123)}\stackrel{C_5}{\sim} \si^1_{(132)}$. \\
So $\mathcal{SL}_4(2) / C_5$ is generated by the five elements $\si^1_{id}$, $\si^1_{(12)}$, $\si^1_{(123)}$, $\si^1_{(13)}$ and $\si^2_{id}$.   
\begin{center}
\begin{figure}[!h]
\includegraphics{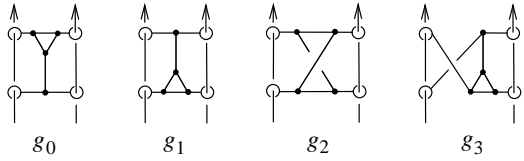}
\caption{The $C_4$-graphs $g_i$ ; $i=0,1,2,3$.} \label{c4tr}
\end{figure}
\end{center}

We introduced in Subsection \ref{statec4} three different ways of closing a $2$-string link $\si$ into a knot, namely by taking the plat closure $\overline{\si}$ and by taking the closure $cl_0$ (resp. $cl_1$) in Figure \ref{threeclose} of the $3$-string link $\Delta_1\sigma$, resp. $\Delta_2\sigma$, obtained from $\si$ by doubling the first, resp. second component.   
In particular if $\si \in \mathcal{SL}_4(2) / C_5$, the resulting knot is an element of $\mathcal{SL}_4(1) / C_5$, and can be expressed in terms of the generators $A$ and $B$ given in Subsection \ref{subknot}.  We collect the results in Table \ref{t1}.  This is straightforward for the plat closure case, and uses the fact that the knot $C$ given in Figure \ref{knots} satisfies $C\stackrel{C_5}{\sim} U$.  (This fact is also used for Table \ref{t3}.)  
For the two latter cases, the computations make use of Calculus of Claspers, and in particular it makes use of Lemma \ref{calculus}(3).  
\begin{table}[!h]
\begin{center}
\begin{tabular}{|c|ccccc|}
\hline
      $\si$            & $\si^1_{id}$ & $\si^1_{(12)}$ & $\si^1_{(123)}$ & $\si^1_{(13)}$ & $\si^2_{id}$ \\[0.1cm] 
\hline
$\overline{\si} / C_5$       &     $A$    &      $U$       &      $B$      &      $B$     &     $A$     \\[0.1cm]  
$cl_0(\Delta_1\sigma) / C_5$ &$A^3\cdot B^{-1}$& $B^{-1}$&     $B^3$     &$A^{-1}\cdot B^3$&$A^4\cdot 2B$   \\[0.1cm]  
$cl_1(\Delta_2\sigma) / C_5$ & $A^4\cdot B^2$& $A\cdot B^2$ & $A\cdot B^6$  & $B^6$ &$A^3\cdot B^{-1}$    \\
\hline	 
\end{tabular}\\[0.1cm]
\caption{ } \label{t1}
\end{center}
\end{table}
For example, let us explain here the computation for $cl_0 (\Delta_1\si^1_{id})$.  Observe that $s^1_{id}$ is precisely the $C_4$-tree represented on the left-hand side of Figure \ref{proofc4fig}.  So, doubling the first component of $\1_2$ yields the $C_3$-tree $s'$ for $\1_3$ represented in this figure, which as shown in Subsection \ref{add} satisfies $(\1_3)_{s'}\stackrel{C_5}{\sim} \prod_{1\le i\le 4} (\1_3)_{s_i}$, where $s_i$ is as illustrated in Figure \ref{proofc4fig} ($i=1,2,3,4$).  
Using Claim \ref{claimadditivity}, we have 
 $$ cl_0\left(\prod_{i=1}^4 (\1_3)_{s_i}\right)\stackrel{C_5}{\sim} \prod_{i=1}^4 cl_0\left((\1_3)_{s_i}\right). $$
By an isotopy, we see that $cl_0\left((\1_3)_{s_1}\right)=U_{k'}$, where $k'$ is a $C_4$-tree for $U$ represented in Figure \ref{proofc4_3}.  As shown there, we have $U_{k'}\stackrel{C_5}{\sim} (U_{s_{(13)}})^{-1}=B^{-1}$ using Lemma \ref{calculus}(2).  
\begin{center}
\begin{figure}[!h]
\includegraphics{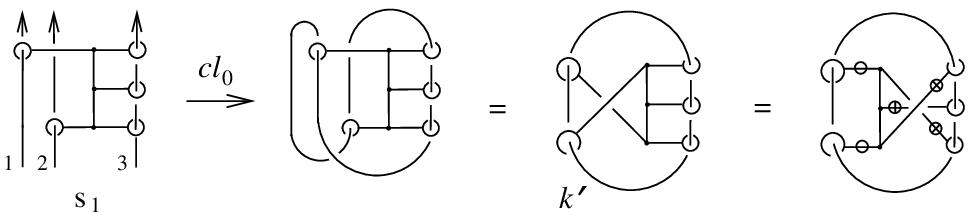}
\caption{The proof that $cl_0\left((\1_3)_{s_1}\right)\stackrel{C_5}{\sim} B^{-1}$.} \label{proofc4_3}
\end{figure}
\end{center}
For $i=2,3,4$, a simple isotopy shows that $cl_0\left((\1_3)_{s_i}\right)=A$.  

Recall that $\mathbf{a}=P^{(4)}_0{(A)}$ and $\mathbf{b}=a_4{(B)}$.  Table \ref{t2} follows immediately from the definitions of the invariants $f^i_4$ ($i=1,...,5$) and the computations given in Table \ref{t1}.  
\begin{table}[!h]
\begin{tabular}{|c|ccccc|}
\hline
      $\si$              & $\si^1_{id}$ & $\si^1_{(12)}$ & $\si^1_{(123)}$ & $\si^1_{(13)}$ & $\si^2_{id}$ \\[0.1cm] 
\hline
$f^1_4(\si) / \mathbf{b}$&      0     &      0       &      $1$      &     $1$      &      0      \\[0.1cm]  
$f^2_4(\si) / \mathbf{a}$&     $1$    &      0       &       0       &      0       & $1$\\[0.1cm] 
$f^3_4(\si) / \mathbf{b}$&    $-1$    &    $-1$      &      $3$      &     $3$      &$2$\\[0.1cm] 	 
$f^4_4(\si) / \mathbf{a}$&     $3$    &      0       &       0       &     $-1$     &$4$\\[0.1cm]  
$f^5_4(\si) / \mathbf{a}$&     $4$    &     $1$      &      $1$      &       0      &$3$\\[0.1cm] 
\hline	 
\end{tabular}\\[0.1cm]
\caption{} \label{t2}
\end{table}
The $5\times 5$ matrix given by the entries of Table \ref{t2} having rank $5$, we obtain that the five invariants $f^i_4$ ($i=1,...,5$) (together with the knot invariants $a_4$ and $P^{(3)}_0$) do classify the abelian group $\mathcal{SL}_4(2) / C_5$, thus completing the proof of the lemma. 
\end{proof}

\begin{remark}
The number of linearly independent finite type $2$-string link invariants of degree $4$ has been computed by Bar-Natan \cite{BNcalc}.  In particular, there are $10$ linearly independent such invariants which do not have a factor coming from a single knot component, see \cite[\S 2.3.4]{BNcalc}.
Half of them come from products of lower degrees invariants (namely $\mu(12)^4$, $\mu(12)^2 f_2$, $\mu(12) f_3$, $\mu(12) \mu(1122)$ and $(f_2)^2$) and the remainning five are the invariants $f^i_4$ ($i=1,...,5$).  
%
\end{remark} 
\subsubsection{The $3$-component case}\label{sub3comp}
In this subsection we prove the following lemma.   
\begin{lemma}\label{claimc4_3}
Let $\si$, $\si'\in \mathcal{SL}_4(3)$. Then $\si$ and $\si'$ are $C_{5}$-equivalent if and only if they share all knots invariants of degree $4$, all invariants $f^i_4$ ($1\le i\le 5$), all invariants $V^j_4$ ($1\le j\le 7$), and all Milnor invariants $\mu(iiijk)$ and $\mu(ijjkk)$ ($1\le i,j,k\le n$ ; $j<k$). 
\end{lemma}
\begin{proof}[Proof of Lemma \ref{claimc4_3}]
Let $\si\in \mathcal{SL}_4(3)$.   
By Calculus of Claspers and subsections \ref{subknot} and \ref{sub2comp} above, 
\begin{equation}\label{eq3comp}
\si\stackrel{C_5}{\sim} \tilde{\si}\cdot (\mathbf{1}_3)_{E}\cdot (\mathbf{1}_3)_{F}, 
\end{equation}
where $\tilde{\si}$ is determined uniquely by the invariants $a_4$, $P_0^{(4)}$ and $f^i_4$ ($1\le i\le 5$) of $\si$, and where $E$, resp. $F$, is a disjoint union of simple $C_4$-trees for $\mathbf{1}_3$ with index $\{i,j^{(2)},k^{(2)}\}$, resp.  $\{i,j,k^{(3)}\}$ ($1\le i,j,k\le n$).  

For $\alpha\in S_3$, denote by $U_\alpha$, $U'_\alpha$, $U''_\alpha$  $V_\alpha$, $V'_\alpha$ and $V''_\alpha$ the $3$-string links obtained from $\mathbf{1}_3$ by surgery along the $C_4$-trees $u_\alpha$, $u'_\alpha$, $u''_\alpha$, $v_\alpha$, $v'_\alpha$ and $v''_\alpha$ represented in Figure \ref{figc4_3}.  
 \begin{figure}[!h]
  \includegraphics{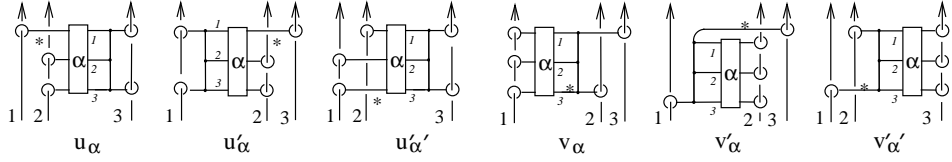}
  \caption{The $C_4$-trees $u_\alpha$, $u'_\alpha$, $u''_\alpha$, $v_\alpha$, $v'_\alpha$ and $v''_\alpha$.} \label{figc4_3}
 \end{figure} 

Set $\mathbb{U}:=\{U_\alpha, U'_\alpha, U''_\alpha \textrm{ $|$ } \alpha\in S_3\}$ and  $\mathbb{V}:=\{V_\alpha, V'_\alpha, V''_\alpha \textrm{ $|$ } \alpha\in S_3\}$.  
By Calculus of Claspers and the AS and IHX relations, we have that the $C_5$-equivalence class of $(\mathbf{1}_3)_E$, resp. of $(\mathbf{1}_3)_F$, is generated by $\mathbb{U}$, resp. by $\mathbb{V}$.  So (\ref{eq3comp}) can be rewritten as 
\begin{equation} \label{eq3comp_2}
\si\stackrel{C_5}{\sim} \tilde{\si}\cdot \prod_{\alpha\in S_3} (U_\alpha)^{m_\alpha}\cdot (U'_\alpha)^{m'_\alpha}\cdot (U''_\alpha)^{m''_\alpha}\cdot (V_\alpha)^{n_\alpha}\cdot (V'_\alpha)^{n'_\alpha}\cdot (V''_\alpha)^{n''_\alpha}
\end{equation}
for some integers $m_\alpha$, $m'_\alpha$, $m''_\alpha$, $n_\alpha$, $n'_\alpha$ and $n''_\alpha$.

We first consider the set $\mathbb{V}\subset \mathcal{SL}_4(3)$.  We have the following
\begin{claim}\label{claim}
Any element $L$ of $\mathbb{V}$ satisfies $L\stackrel{C_5}{\sim} L^{(1)}\cdot L^{(2)}$, where $L^{(1)}$ is obtained from $\mathbf{1}_3$ by surgery along $C_4$-trees of index $\{i,j^{(2)},k^{(2)}\}$ and where $L^{(1)}$ is generated by the elements 
\beQ
I_1:=V_{id}, & I_2:=V'_{id}, & I_3:=V''_{id}.
\eeQ
\end{claim}
\begin{proof}[Proof of Claim \ref{claim}]
Let us consider the case of the $6$ elements $V_\alpha$ ($\alpha\in S_3$).  
For $\alpha$ = $(12)$, the STU relation gives 
$V_{(12)} \stackrel{C_5}{\sim} I_1\cdot (\mathbf{1}_3)_{c_4}$, where $c_4$ is the $C_4$-graph represented in Figure \ref{figc4_4}.  The claim thus follows, by using the STU relation to express $(\mathbf{1}_3)_{c_3}$ as a product $(\mathbf{1}_3)_T\cdot (\mathbf{1}_3)_{T'}$ for two $C_4$-trees $T$ and $T'$ with index $\{1^{(2)},2,3^{(2)}\}$.  
The same argument can be applied for any $\alpha\in S_3$, as $V_\alpha$ is related to $V_{id}$ by successive applications of the STU relation.  
By symmetry, the case of elements $V'_\alpha$ and $V''_\alpha$ is also strictly similar.  
\end{proof}

Now, observe that the $3$-string links $I_l$, $l=1,2,3$, are distinguished by Milnor invariants.  More precisely, for all $1\le i,j,k\le 3$ with $j<k$, we have $\mu_{I_l}(iiijk) = \pm \delta_{i,l}$.  Note also that $\mu_{I_l}(ijjkk) = 0$.   

So there remains to classify the $18$ elements of $\mathbb{U}\subset \mathcal{SL}_4(3)$.  
The following relations among elements of $\mathbb{U}$ can be proved using the AS, IHX and STU relations.  
\beQ
U_{(123)}\cdot (U_{(13)})^{-1} & \stackrel{C_5}{\sim} & U''_{id}\cdot (U''_{(23)})^{-1}, \\ 
U_{id}\cdot (U_{(23)})^{-1} & \stackrel{C_5}{\sim} & U''_{(123)}\cdot (U''_{(13)})^{-1}, \\ 
U'_{(123)}\cdot (U'_{(13)})^{-1} & \stackrel{C_5}{\sim} & (U''_{id})^{-1}\cdot U''_{(12)}\cdot U''_{(132)}\cdot (U''_{(13)})^{-1}, \\ 
U'_{id}\cdot (U'_{(23)})^{-1} & \stackrel{C_5}{\sim} & (U''_{(123)})^{-1}\cdot U''_{(12)}\cdot U''_{(132)}\cdot (U''_{(23)})^{-1}, \\ 
U_{id}\cdot (U_{(12)})^{-1}\cdot (U_{(132)})^{-1}\cdot U_{(13)} & \stackrel{C_5}{\sim} & U'_{(123)}\cdot (U'_{(12)})^{-1}\cdot (U'_{(132)})^{-1}\cdot U'_{(23)}. 
\eeQ
More precisely, the first relation is obtained as follows.  Consider the $C_4$-graph $c_0$ represented in Figure \ref{figc4_4}.  By applying the STU relation to the edge incident to the $1$-leaf of $c_0$, we have $(\1_n)_{c_0}\stackrel{C_5}{\sim} (U''_{id})^{-1}\cdot U''_{(23)}$.  Now, it follows from the IHX relation that $(\1_n)_{c_0}\stackrel{C_5}{\sim} (\1_n)_{c_1}$, where $c_1$ is shown in Figure \ref{figc4_4}, and on the other hand the STU relation can be used to show that $(\1_n)_{c_1}\stackrel{C_5}{\sim} (U_{(123)})^{-1}\cdot U_{(13)}$, which implies the desired relation.  The next four relations are proved strictly similarly by using respectively the $C_4$-graphs $c_2$, $c_3$, $c_4$ and $c_5$ of Figure \ref{figc4_4} in place of $c_1$.  
 \begin{figure}[!h]
  \includegraphics{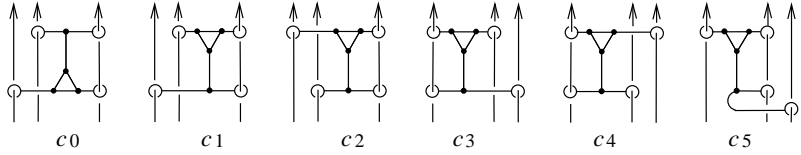}
  \caption{The $C_4$-graphs $c_i$, $1\le i\le 5$.} \label{figc4_4}
 \end{figure} 

Also, we can use the $C_4$-graphs $d_1$ and $d_2$ of Figure \ref{figc4_5} in a similar way (that is, by applying the STU relation in two different way) to obtain the additional two relations. 
\beQ
U''_{(12)}\cdot U''_{(132)} & \stackrel{C_5}{\sim} & U_{(12)}\cdot U_{(132)} \\
U''_{(12)}\cdot (U''_{(132)})^{-1} & \stackrel{C_5}{\sim} & U'_{(12)}\cdot (U'_{(132)})^{-1} , 
\eeQ

Finally let us show that 
\begin{equation}\label{equU}
U''_{(132)} \stackrel{C_5}{\sim} (U_{(12)})^{-1}\cdot U'_{(12)}.  
\end{equation}
To prove (\ref{equU}), we need the following lemma, which can be easily derived from the proof of \cite[Prop. 4.4]{H}. 
\begin{lemma}\label{remstu}
Let $G_S$ be a $C_k$-graph for $\1_n$, and let $G_T$ and $G_U$ be the unions of two \emph{tree} claspers which differ from $G_S$ only in a small ball as depicted in Figure \ref{asihxstu_fig}, where the two leaves of $G_T$, resp. $G_U$, are from different components.  
Then $(\1_n)_{G_S} \stackrel{C_{k+1}}{\sim} \left( (\1_n)_{G_T}\right)^{-1}\cdot (\1_n)_{G_U}$,  
where $\left((\1_n)_{G_T}\right)^{-1}$ denotes the (formal) inverse of $(\1_n)_{G_T}$ in the abelian group  $\mathcal{SL}_k(n) / C_{k+1}$.  
\end{lemma}
Observe that, by the AS relation and Calculus of Claspers, we have $U''_{(132)} \stackrel{C_5}{\sim} \1_T$, where $T$ is the $C_4$-tree represented in Figure \ref{figc4_5}.  
 \begin{figure}[!h]
  \includegraphics{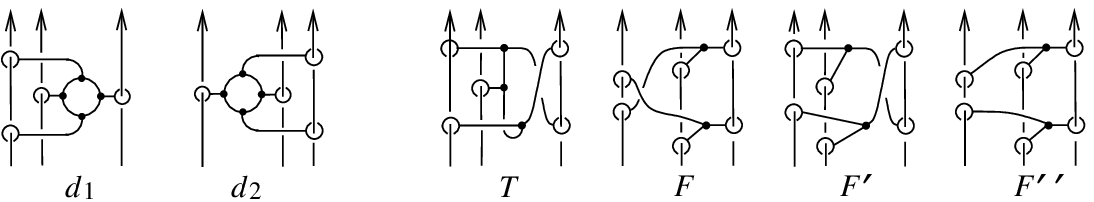}
  \caption{
  } \label{figc4_5}
 \end{figure} 
We have 
\[ \1_T \stackrel{C_5}{\sim} \1_F\cdot \left(\1_{F'}\right)^{-1}\stackrel{C_5}{\sim} \1_F\cdot (\1_{F''})^{-1}\cdot \1_{F''}\cdot \left(\1_{F'}\right)^{-1}\stackrel{C_5}{\sim} (U_{(12)})^{-1}\cdot U'_{(12)}, \]
where $F$, $F'$ and $F''$ are as shown in Figure \ref{figc4_5}.  Here, the first and third equivalence follow from Lemma \ref{remstu} and isotopies, and $(\1_{F'})^{-1}$, resp. $(\1_{F''})^{-1}$, denotes the inverse of $\1_{F'}$, resp. $\1_{F''}$, in $\mathcal{SL}(4) / C_5$.  

So we obtain that $\mathbb{U}$ is generated by the following $10$ elements: $U''_{id}$, $U''_{(12)}$, $U''_{(123)}$, $U''_{(23)}$, $U''_{(13)}$, $U_{id}$, $U_{(12)}$, $U_{(123)}$, $U'_{id}$ and $U'_{(12)}$.  In order to show that they are linearly independent, we make use of Milnor invariants $\mu(ijjkk)$ ($1\le i,j,k\le 3$ ; $j<k$) and the invariants $V^i_4$ ($i=1,...,7$) defined in Subsection \ref{statec5}.  We compute the $C_5$-equivalence classes of the closures $cl_j$ ($j=1,...,4$) of the $10$ elements listed above in a similar way as for Table \ref{t1}.  These computations are summarized in Table \ref{t3}.  
\begin{table}[!h]
\begin{tabular}{|c|ccccc|ccc|cc|}
\hline
$\si$       & $U''_{id}$ & $U''_{(12)}$ & $U''_{(123)}$ & $U''_{(23)}$ & $U''_{(13)}$ & $U_{id}$ & $U_{(12)}$ & $U_{(123)}$ & $U'_{id}$ & $U'_{(12)}$ \\[0.1cm] 
\hline
$cl_1(\si) / C_5$ & $A^{-1}$ & $U$ & $B^{-1}$ & $U$ & $B^{-1}$ & $B$ & $B$ & $U$ & $B$ & $U$ \\[0.1cm]  
$cl_2(\si) / C_5$ & $B$ & $U$ & $A$ & $B$ & $U$ & $A$ & $U$ & $B$ & $U$ & $B$ \\[0.1cm]  
$cl_3(\si) / C_5$ & $U$ & $B$ & $B$ & $A$ & $B$ & $B^{-1}$ & $U$ & $A^{-1}$ & $A^{-1}$ & $U$ \\[0.1cm]  
$cl_4(\si) / C_5$ & $B$ & $B$ & $U$ & $B$ & $A$ & $A^{-1}$ & $U$ & $B^{-1}$ & $B^{-1}$ & $U$ \\
\hline	 
\end{tabular}\\[0.1cm]
\caption{ 
 } \label{t3}
\end{table}

From Table \ref{t3} and the definitions of the invariants, we obtain the desired computations, as given in Table \ref{t4}.  The matrix given by this table has rank $10$, which shows that any element generated by $\mathbb{U}$ is uniquely determined by the invariants listed in the table. 
\begin{table}[!h]
\begin{tabular}{|c|ccccc|ccc|cc|}
\hline
$\si$       & $U''_{id}$ & $U''_{(12)}$ & $U''_{(123)}$ & $U''_{(23)}$ & $U''_{(13)}$ & $U_{id}$ & $U_{(12)}$ & $U_{(123)}$ & $U'_{id}$ & $U'_{(12)}$ \\[0.1cm] 
\hline
$\mu_{\si}(12233)$ & $0$ & $0$ & $0$ & $0$ & $0$ & $1$ & $0$ & $-1$ & $0$ & $0$ \\[0.1cm]  
$\mu_{\si}(32211)$ & $0$ & $0$ & $0$ & $0$ & $0$ & $0$ & $0$ & $0$ & $-1$ & $0$ \\[0.1cm]  
$\mu_{\si}(21133)$ & $1$ & $0$ & $-1$ & $1$ & $-1$ & $0$ & $0$ & $0$ & $0$ & $0$ \\[0.1cm]  
\hline	
$V^1_4(\si) / \mathbf{b}$ & $0$ & $0$ & $-1$ & $0$ & $-1$ & $1$ & $1$ & $0$ & $1$ & $0$ \\[0.1cm]  
$V^2_4(\si) / \mathbf{b}$ & $1$ & $0$ & $0$ & $1$ & $0$ & $0$ & $0$ & $1$ & $0$ & $1$ \\[0.1cm]  
$V^3_4(\si) / \mathbf{b}$ & $0$ & $1$ & $1$ & $0$ & $1$ & $-1$ & $0$ & $0$ & $0$ & $0$ \\[0.1cm] 
$V^4_4(\si) / \mathbf{a}$ & $-1$ & $0$ & $0$ & $0$ & $0$ & $0$ & $0$ & $0$ & $0$ & $0$ \\[0.1cm]  
$V^5_4(\si) / \mathbf{a}$  & $0$ & $0$ & $1$ & $0$ & $0$ & $1$ & $0$ & $0$ & $0$ & $0$ \\[0.1cm]  
$V^6_4(\si) / \mathbf{a}$ & $0$ & $0$ & $0$ & $1$ & $0$ & $0$ & $0$ & $-1$ & $-1$ & $0$ \\[0.1cm] 
$V^7_4(\si) / \mathbf{a}$  & $0$ & $0$ & $0$ & $0$ & $1$ & $-1$ & $0$ & $0$ & $0$ & $0$ \\
\hline	 
\end{tabular}\\[0.1cm]
\caption{ } \label{t4}
\end{table}
It follows that Milnor invariants $\mu_{\si}(iiijk)$ and $\mu_{\si}(ijjkk)$ ($1\le i,j,k\le 3$ , $j<k$), and the invariants $V^i_4$ ($1\le i\le 7$) of $\si$ determine uniquely all the exponents in (\ref{eq3comp_2}).  
The lemma then follows from the $4$-additivity of these invariants.  
\end{proof}
\subsubsection{The $4$-component case}\label{sub4comp}
Let $\si\in \mathcal{SL}_4(4)$.  We proceed as in the previous subsections to construct a representative of the $C_5$-equivalence class of $\s$. 

By Calculus of Claspers and subsections \ref{subknot} to \ref{sub3comp}, we have 
$$ \si\stackrel{C_5}{\sim} \tilde{\si}\cdot (\mathbf{1}_4)_{P} $$ 
where $\tilde{\si}$ is uniquely determined by the invariants of $\si$ listed in Lemma \ref{claimc4_3}, and where $P$ is a disjoint union of simple $C_4$-trees for $\mathbf{1}_4$ with index $\{i,j,k,l^{(2)}\}$ ($1\le i,j,k,l\le 4$).  By the IHX relation, we may assume that each $C_4$-tree in $P$ is linear and that its ends are the two $l$-leaves.  
Recall from Subsection \ref{calculus} that $\mathcal{B}_4(k)$ ($1\le k\le 4$) is the set of all bijections $\tau$ from $\{ 1,...,3 \}$ to $\{1,...,4\} \setminus \{k \}$ such that $\tau(1)<\tau(3)$.  
By Lemma \ref{flipping}, there exists integers $m_\alpha$ ($\alpha\in \mathcal{B}_4(4)$) and $m_{\alpha,k}$ ($1\le k\le 4$ and $\alpha\in \mathcal{B}_4(k)$) such that 
\begin{equation}\label{c4equation}
 (\1_4)_P\stackrel{C_5}{\sim} \prod_{\alpha\in \mathcal{B}_4(4)} (\ov{B_{\alpha}}(4))^{m_{\alpha}}\cdot \prod_{k=1}^4 \prod_{\alpha\in \mathcal{B}_4(k)} (B_{\alpha}(k))^{m_{\alpha,k}}, 
\end{equation}
where the string links $B_{\alpha}(k)$ and $\ov{B_{\alpha}}(k)$ are defined in Figure \ref{flipping_fig}.  

For $\tau\in \mathcal{B}_4(k)$, set $\mu_{\tau} := \mu(\tau(1),\tau(2),\tau(3),k,k)$.  Then for any $1\le l\le 4$ and $\eta \in \mathcal{B}_4(l)$, we have (see \cite[\S 4]{MY}): 
 $$\mu_{\tau} (B_{\alpha}(k))=\mu_{\tau} (\ov{B_{\alpha}}(k))=\delta_{\alpha,\tau}.  $$
Observe that, by definition, $\mathcal{B}_4(4)$ is just the subgroup $\{ id,(12),(23)\}$ of $S_3$.  
One can check that the closures $K_iC_j$ and $K_iC'_j$ of the six $4$-string links $B_{\alpha}(4)$, $\ov{B_{\alpha}}(4)$ ($\alpha\in \mathcal{B}_4(4)$) are given in Table \ref{t5} ($1\le i,j\le 3$).  
\begin{table}[!h]
\begin{tabular}{|c|cccccc|}
\hline
$\si$ & $B_{id}(4)$ & $B_{(12)}(4)$ & $B_{(23)}(4)$ & $\ov{B_{id}}(4)$ & $\ov{B_{(12)}}(4)$ & $\ov{B_{(23)}}(4)$ \rule[-7pt]{0pt}{20pt} \\[0.1cm] 
\hline
$K_1(\si) / C_5$ & $B^{-1}$ & $B^{-1}$ & $B^{-1}$ & $A$ & $U$ & $U$ \\[0.1cm]  
$K_2(\si) / C_5$ & $B^{-1}$ & $B^{-1}$ & $U$ & $U$ & $A$ & $B$ \\[0.1cm]  
$K_3(\si) / C_5$ & $B^{-1}$ & $U$ & $B^{-1}$ & $U$ & $B$ & $A$ \\
\hline	 
\end{tabular}\\[0.1cm]
\caption{ 
 } \label{t5}
\end{table}
We thus obtain the values of the invariants $\mu_{\tau}$ ($\tau\in \mathcal{B}_4(4)$) and  $W^i_4$ ($i=1,2,3$) as in Table \ref{t6}
\begin{table}[!h]
\begin{tabular}{|c|cccccc|}
\hline 
$\si$       & $B_{id}(4)$ & $B_{(12)}(4)$ & $B_{(23)}(4)$ & $\ov{B_{id}}(4)$ & $\ov{B_{(12)}}(4)$ & $\ov{B_{(23)}}(4)$ \rule[-7pt]{0pt}{20pt} \\[0.1cm] 
\hline
$\mu_\si(12344)$ & $1$ & $0$ & $0$ & $1$ & $0$ & $0$ \\[0.1cm]  
$\mu_\si(21344)$ & $0$ & $1$ & $0$ & $0$ & $1$ & $0$ \\[0.1cm]  
$\mu_\si(13244)$ & $0$ & $0$ & $1$ & $0$ & $0$ & $1$ \\[0.1cm] 
\hline	 
$W^1_4(\si) / \mathbf{a}$ & $0$ & $0$ & $0$ & $1$ & $0$ & $0$ \\[0.1cm]  
$W^1_4(\si) / \mathbf{a}$ & $0$ & $0$ & $0$ & $0$ & $1$ & $0$ \\[0.1cm]  
$W^1_4(\si) / \mathbf{a}$ & $0$ & $0$ & $0$ & $0$ & $0$ & $1$ \\
\hline	 
\end{tabular}\\[0.1cm]
\caption{ 
 } \label{t6}
\end{table}
Clearly, the matriw given by the entries of Table \ref{t6} has rank $6$.  
This implies that all exponents in (\ref{c4equation}) are uniquely determined by Milnor invariants $\mu_\si(jikll)$ ($1\le i,j,k,l\le 4$ ; $j<k$) and all invariants $W^i_4$ of $\si$ ($i=1,2,3$), and thus proves the result by the $4$-additivity of these invariants.  
\subsubsection{Proof of Theorem \ref{c5}}\label{subproof}
We now prove Theorem \ref{c5} in the general case.  

 Given $\si\in \mathcal{SL}(n)$, we know from the proof of Theorem \ref{c3} that $\si$ is $C_4$-equivalent to $\si_{(2)}\cdot \si_{(3)}$, where $\si_{(2)}$ and $\si_{(3)}$ are given by (\ref{L2}) and (\ref{L3}) respectively. 

 By Calculus of Claspers, $\si\stackrel{C_5}{\sim} \si_{(2)}\cdot \si_{(3)}\cdot \si_{(4)}$, where $\si_{(4)}$ is obtained from $\1_n$ by surgery along a union of $C_4$-trees.  More precisely, 
  $$ \si_{(4)}\stackrel{C_5}{\sim} \prod_{i=1}^{5}\si^i_{(4)}, $$
where for each $i=1,...,5$, the $n$-string link $\si^i_{(4)}$ is obtained from $\1_n$ by surgery along a union of $C_4$-trees that each intersect $i$ distinct components of $\1_n$.  

By subsections \ref{subknot} to \ref{sub4comp}, we can determine explicitly $\si^i_{(4)}$ for $1\le i\le 4$ using all invariants $a_4$, $P^{(4)}_0$, $f^i_4$ ($1\le i\le 5$), $V^j_4$ ($1\le j\le 7$) and $W^k_4$ ($1\le k\le 3$) of $\si$, and all Milnor invariants $\mu_\si(iiijk)$, $\mu_\si(ijjkk)$ and $\mu_\si(jikll)$ ($1\le i,j,k,l\le n$ ; $j<k$).  (Using the fact that all these invariants are $4$-additive).  

Now, it is easy to see that Milnor invariants $\mu(ijklm)$ ($1\le i,j,k<l<m\le n$) do classify $n$-string links of the form $(\1_n)_{T}$ for $T$ a $C_4$-tree intersecting $5$ distinct components of $\1_n$.  Indeed, if $T$ has index $I=\{i,j,k,l,m\}$ ($1\le i,j,k<l<m\le n$), we may assume by the IXH relation that $T$ is linear, and that the ends are the $l$-leaf and $m$-leaf.  Then for every multi-index $I'=i'j'k'l'm'$ ($1\le i',j',k'<l'<m'\le n$) we have $\mu_{\1_{T}}(I') = \pm \delta_{I,I'}$, see \cite{Milnor2,HM}.  
Since these Milnor invariants are $4$-additive, the proof is completed.  
\section{Finite type concordance invariants}\label{fticonc}
In this section, we define the equivalence relation on string links generated by $C_k$-moves and concordance, called $C_k$-concordance.  We show that finite type concordance invariants classify string links up to $C_k$-concordance for $k\le 6$.  
\subsection{$C_k$-concordance}
Recall that two $n$-string links $\sigma, \sigma'$ are \emph{concordant} if there is an embedding 
$$ f: \left(\sqcup_{i=1}^n I_i\right)\times I \longrightarrow \left(D^2\times I\right)\times I $$  
such that $f\left( (\sqcup_{i=1}^n I_i)\times \{ 0 \} \right)=\sigma$ and $f\left( (\sqcup_{i=1}^n I_i)\times \{ 1 \} \right)=\sigma'$, and such that $f\left(\partial(\sqcup_{i=1}^n I_i)\times I\right)=(\partial \sigma) \times I$.   
String link concordance is an equivalence relation, and is denoted by $\stackrel{c}{\sim}$. 

In order to study finite type concordance invariants, it is natural to consider the following.  
\begin{definition}
Let $k,n\ge 1$ be integers.  
Two $n$-string links $\sigma, \sigma'$ are \emph{$C_k$-concordant} if there is a sequence 
$\sigma=\sigma_0,\sigma_1, ... ,\sigma_m=\sigma'$ such that 
 for each $i\ge 1$, either $\sigma_i\stackrel{C_{k+1}}{\sim} \sigma_{i+1}$ 
or $\sigma_i\stackrel{c}{\sim} \sigma_{i+1}$.  
We denote the $C_k$-concordance relation by $\stackrel{C_k+c}{\sim}$. 
\end{definition} 

Clearly, two $C_k$-concordant string links share all finite type concordance invariants of degree less than $k$.   
It is thus natural to ask the following.  
\begin{question} \label{question}
Let $\sigma$, $\sigma'\in \mathcal{SL}(n)$.  Do we have 
\[ \textrm{$\sigma\stackrel{C_k+c}{\sim} \sigma'$ $\Leftrightarrow$ 
They share all finite type concordance invariants of degree $<k$ ?} \]
\end{question}
We give a positive answer to this question for $k\le 6$ in Subsection \ref{ckcsl}.  

It is known that Milnor invariants are concordance invariants \cite{casson}.  
So by \cite[Thm. 7.1]{H}, $\mu(J)$ is a $C_k$-concordance invariant for any $J$ with $|J|\leq k$. 
Habegger and Masbaum showed that all \emph{rational} finite type concordance invariants of 
string links are given by Milnor invariants via the Kontsevich integral \cite{HMa}.  
\subsection{The ordered index}
In order to study $C_k$-concordance for string links, we use the notion of ordered index of a $C_k$-tree.  

\begin{definition}\label{order-index} 
Let $t$ be a linear $C_k$-tree with ends $f_0,f_k$. 
Since $t$ is a disk, we can travel from $f_0$ to $f_k$ along $\partial t$ so that 
we meet all other leaves $f_1,...,f_{k-1}$ in this order.
If $f_s$ is an $i_s$-leaf $(s=0,...,k)$, we can consider two vectors
$(i_0,...,i_k)$ and $(i_k,...,i_0)$ and may assume that 
$(i_0,...,i_k)\leq(i_k,...,i_0)$, where \lq$\leq$\rq~  is the lexicographic order in ${\Bbb Z}^{k+1}$.
We call $(i_0,...,i_k)$ the {\em ordered index} of $t$ and denote it by o-index$(t)$. 
\end{definition}

By Calculus of Claspers and AS, IHX, STU relations, we have the following.
\begin{lemma} \label{o-index}
(1)~Let $t$ and $t'$ be linear $C_k$-trees for $\1_n$ 
with same ordered index.  Then there are $C_k$-graphs $g_1,...,g_m$ with loops such that $(\1_n)_{t'}\stackrel{C_{k+1}}{\sim} (\1_n)_t^{\varepsilon}\cdot\prod_i(\1_n)_{g_i}$
for some $\varepsilon=\pm 1$.

(2)~Let $t$ be a linear $C_k$-tree $(k\geq 3)$ for $\1_n$ with o-index$(t)=(i_0,...,i_k)$. 
If $i_0=i_1$ or $i_{k-1}=i_k$, then there are $C_k$-graphs $g_1,...,g_m$ with loops such that 
$(\1_n)_{t}\stackrel{C_{k+1}}{\sim} \prod_i(\1_n)_{g_i}$.

(3)~Let $t$ be a linear $C_k$-tree $(k\geq 2)$ for $\1_n$ with o-index$(t)=(i_0,...,i_k)$.  If $(i_0,...,i_k)=(i_k,...,i_0)$ and $k$ is even, then there are $C_k$-graphs $g_1,...,g_m$ with loops such that $((\1_n)_{t})^2\stackrel{C_{k+1}}{\sim} \prod_i(\1_n)_{g_i}$.
\end{lemma}
Before proving this lemma, we need the following definition.  
A $C_k$-tree for $\1_n$ is \emph{planar} if it can be represented, 
in the usual diagram of $\1_n$, by a tree clasper without any crossing among the edges and with edges overpassing all 
components of $\1_n$ up to isotopy.  

\begin{proof}[Proof of Lemma \ref{o-index}]
Statements $(1)$ and $(2)$ follow from similar arguments as for Lemma \ref{index1}.  
For $(1)$, observe that $t$ can be deformed into $t'$ by crossing changes and sliding leaves.  
By the STU relation, if $c'$ is obtained from a $C_k$-tree $c$ for $\1_n$ by a sliding a leaf, 
we have $(\1_n)_c \stackrel{C_{k+1}}{\sim} (\1_n)_{c'}\cdot (\1_n)_g$ for some $C_k$-graph $g$ with loop.  
For $(2)$, use the IHX and STU relation as in the proof of Lemma \ref{index1}.  

For simplicity, we show $(3)$ in the case where $t$ is planar and both ends of $t$ are $n$-leaves.  
By assumption the o-index$(t)$ has the form $(i_0,...,i_{k/2-1},i_{k/2},i_{k/2-1},...,i_0)$.  
We may assume that the axis $a$ of the edge incident to the $i_{k/2}$-leaf of $t$ is transverse 
to each component of $\1_n$ up to isotopy.  Let $\tilde{t}$ be obtained by $180$-degree rotation 
of $t$ around $a$ fixing the leaves.  
By sliding the leaves of $\tilde{t}$ repeatedly, we can deform it into a planar $C_k$-tree 
$\overline{t}$ which only differs from $t$ by a half-twist on each edge incident to a leaf.  
By the observation above, the STU relation gives that 
$(\1_n)_{\tilde{t}}\stackrel{C_{k+1}}{\sim} (\1_n)_{\overline{t}}\prod_i(\1_n)_{g_i}$ for 
some union $g_1,...,g_m$ of $C_k$-graphs with loops.  
On the other hand, by Lemma \ref{calculus}(2) we have 
$(\1_n)_{t})\cdot (\1_n)_{\overline{t}}\stackrel{C_{k+1}}{\sim} \1_n$.  
The result follows.  
\end{proof}
It is known that surgery along graphs with loop implies concordance.  
\begin{lemma}(\cite{CT,GL})\label{graph} 
Let $g$ be a $C_k$-graph with loop for $\1_n$.  Then $(\1_n)_g \stackrel{c}{\sim} \1_n$.  
\end{lemma}
There are in general many linear $C_k$-trees with same ordered index, so $T(i_0,...,i_k)$ is not determined by the o-index. 
For each o-index $(i_0,...,i_k)$, we choose one string link $T(i_0,...,i_k)$ obtained from $\1_n$ by surgery along a linear $C_k$-tree with o-index $(i_0,...,i_k)$, and fix it. 
We note that by Lemmas~\ref{o-index}~(1) and \ref{graph}, there are essentially two choices in $\mathcal{SL}(n) / (C_{k+1}+c)$ 
for each o-index, namely $T(i_0,...,i_k)$ and $T(i_0,...,i_k)^{-1}$.  

The next lemma can be obtained using the calculation method in \cite[Rem. 5.3]{akira}.  
\begin{lemma}\label{milnor-inv}
Let $\sigma$ be an $n$-string link obtained from $\1_n$ by surgery along a linear $C_k$-tree with o-index $I=(i_0,...,i_k)$. 
If $\{i_0,i_k\}\cap\{i_1,...,i_{k-1}\}=\emptyset$, then we have:

(1)~If $(i_0,...,i_k)\neq(i_k,...,i_0)$, then for any $J=i_0j_1...j_{k-1}i_k$, $\mu_{\sigma}(J)=\pm \delta_{I,J}$. 

(2)~If $(i_0,...,i_k)=(i_k,...,i_0)$ and $k$ is an odd number $2m+1$, i.e., $(i_0,...,i_k)=(i_0,...,i_{m},i_{m},...,i_0)$, then for any $J=i_0j_1...j_{k-1}i_0$, $\mu_{\sigma}(J)=\pm 2 \delta_{I,J}$. 

(3)~If $(i_0,...,i_k)=(i_k,...,i_0)$ and $k$ is an even number $2m$, i.e., 
$(i_0,...,i_k)=(i_0,...,i_{m-1},i_m,i_{m-1},...,i_0)$, 
then the Milnor invariants of $\sigma$ with length $\leq k$ vanish, and 
for any $J=i_0j_1...j_{2m}i_0$, 
\[\mu_{\sigma}(J)=
\left\{\begin{array}{ll}
\pm1 & \text{if $(j_1,...,j_{2m})=(i_1,...,i_{m-1},i_m,i_m,i_{m-1},...,i_1)$}\\
0 & \text{otherwise}
\end{array}\right.\]
\end{lemma}
\begin{remark} \label{hop}
It follows in particular from $(2)$ that Milnor invariants $\mu(I)$ (mod $2$) with $I=(i_0,...,i_{m},i_{m},...,i_0)$ are $C_{k+1}$-equivalence invariance ($k=2m+1$).  
\end{remark}
\subsection{$C_k$-concordance for knots} \label{ckcknots}
In this Subsection we give a classification of knots up to $C_k$-concordance.  

Recall that $T(iii)$ is a fixed $n$-string link obtained from the trivial 1-string link $\1$ by surgery along a linear $C_2$-tree with o-index $(i,i,i)$.  Note that this tree can be chosen to be the $C_2$-tree $t_i$ represented in Figure \ref{c2tree}, in which case the closure of the $i$th component of $T(iii)$ is the right-handed trefoil. 
For $n=1$, we simply denote $T(111)$ by $T$.  
\begin{lemma}\label{string-knot}
Let $\sigma$ be a $1$-string link. 
For any integer $k\geq 3$, there is a union $G$ of disjoint graph claspers with loops for $\1_n$ such that 
$\sigma\stackrel{C_{k}}{\sim} T^{\varepsilon}\cdot (\1_n)_{G}$ for some $\varepsilon\in\{0,1\}$.  
\end{lemma}
\begin{proof}
We proceed by induction on $k$. 
For $k=3$, by Lemma \ref{o-index}~(1) (or Theorem \ref{c3}), 
we have $\sigma\stackrel{C_{3}}{\sim} T^x$ for some integer $x$.
By Lemma \ref{o-index}~(3), there is a union $g$ of $C_2$-graph with loop such that 
$T^2\stackrel{C_3}{\sim} (\1)_g$.  
(Actually, it is easy to check using the AS and STU relations that in this case $g$ is connected).   
Hence we have 
\[\sigma \stackrel{C_3}{\sim}
\left\{
\begin{array}{ll}
 T\cdot((\1)_g)^{(x+x/|x|)/2} & \text{if $x$ is odd},\\
 ((\1)_g)^{x/2} & \text{if $x$ is even}. 
 \end{array}\right.\]

Now suppose that there is a union $g_1,...,g_m$ of disjoint graph claspers with loops for $\1_n$ such that $\sigma \stackrel{C_k}{\sim} T^{\varepsilon}\cdot\prod_i(\1_n)_{g_i}$.
Hence $\sigma$ is obtained from $T^{\varepsilon}\cdot\prod_i(\1_n)_{g_i}$ by 
surgery along linear $C_k$-trees. 
Since any linear $C_k$-tree for a 1-string link has o-index $(1,...,1)$, by 
Lemma~\ref{o-index}~(2), we have that there are $C_k$-graphs $h_1,...,h_l$ 
with loops for $\1_n$ such that 
$\sigma\stackrel{C_{k+1}}{\sim} T^{\varepsilon}\cdot\prod_i(\1_n)_{g_i}\cdot\prod_j(\1_n)_{h_j}$.
\end{proof}

We prove the following.  

\begin{theorem}
For an integer $k\geq 3$, two knots $K$ and $K'$ are $C_k$-concordant if and only if $\mathrm{Arf}(K)=\mathrm{Arf}(K')$. 
\end{theorem}

Recall that any knot is $C_2$-equivalent to the trivial one \cite{MN}.

\begin{proof}
Let $\sigma$ and $\sigma'$ be 1-string links whose closures are 
$K$ and $K'$ respectively. 
By Lemma~\ref{string-knot}, there are graph claspers $g_1,...,g_m$ and 
$g'_1,...,g'_l$ with loops for $\1_n$ such that 
$\sigma\stackrel{C_{k}}{\sim} T^{\varepsilon}\cdot\prod_i(\1_n)_{g_i}$ 
and 
$\sigma'\stackrel{C_{k}}{\sim} T^{\varepsilon'}\cdot\prod_j(\1_n)_{g'_j}$ 
for some $\varepsilon,\varepsilon'\in\{0,1\}$. 
So by Lemma~\ref{graph}, 
$\sigma\stackrel{C_{k}+c}{\sim} T^{\varepsilon}$ and 
$\sigma'\stackrel{C_{k}+c}{\sim} T^{\varepsilon'}$. 
Since the Arf invariant is a $C_k$-concordance invariant \cite{rob}, and since 
$\mathrm{Arf}(T)=1$, we have 
$\mathrm{Arf}(K)=\varepsilon$ and $\mathrm{Arf}(K')=\varepsilon'$. 
This completes the proof.
\end{proof}
\begin{remark}
This result is also proved in \cite{sase}, using different methods.  
Another (non-direct) proof can also be obtained by combining Theorem \ref{cnknots} and \cite{ng}. 
\end{remark}
\subsection{$C_k$-concordance for string links}\label{ckcsl}
In this section, we give classifications for $n$-string links up to 
$C_k$-concordance $(k=3,4,5,6)$. 
For each $k>0$, the set of $C_k$-concordance classes forms a group. 
In order to give these classifications, we give a representative of the $C_k$-concordance class of an arbitrary $n$-string link in terms of the generators $T(iii)~(1\leq i\leq n)$ and $T(I)$'s, where $I$ contains at least 2 distinct 
integers. More precisely, we will show that any string link is 
$C_k$-concordant to $\prod_i T(iii)^{a(i)}\cdot\prod_I T(I)^{b(I)}$ 
where $a(i)$ and $b(I)$ are determined by the Arf invariant 
and (mod 2) Milnor invariants respectively. 
For $k=3,4,5$, we already have generators for the $C_k$-equivalent classes, by the proofs of Theorems \ref{c3}, \ref{c4} and \ref{c5}, and we can choose the desired generators among them.  We will introduce similar generators for $k=6$.

We will give the classification results successively, as consequences of each step of our construction of a representative of the $(C_6+c)$-equivalence class of a string link.  In particular, the various proofs are contained in this construction.  

Before starting the construction, we fix the convention below. 
\begin{convention} 
(1). By Lemma~\ref{milnor-inv}~(1) and (2), we see that for each o-index $I=(i_0,...,i_k)$ with  
$\{i_0,i_k\}\cap\{i_1,...,i_{k-1}\}=\emptyset$, we have $\mu_{T(I)}(I)=\pm 1$ or $\pm2$. 
As mentioned before, we have essentially two choices for $T(I)$ and $T(I)^{-1}$ up to $C_{k+1}$-concordance. 
In this section, we chose $T(I)$ so that $\mu_{T(I)}(I)$ is positive whenever $I$ satisfies Lemma~\ref{milnor-inv}~(1) or (2).   (Note that for such a multi-index $I$ we have $\mu_{T(I)^{-1}}(I)=-\mu_{T(I)}(I)$).  
For example, $T(ij)$ is the $n$-string link $L_{ij}$ obtained from $\1_n$ by surgery along the $C_1$-tree $l_{ij}$ of Figure \ref{c2tree} ($1\le i<j\le n$).  

(2). When denoting o-indices, we will let distinct letters denote distinct integers unless otherwise specified. For example 
the set $\{(ijk)|1\leq i,j,k\leq n\}$ of o-indices does not contain $(iii)~(1\leq i\leq n)$. 
\end{convention}

Let $\sigma$ be an $n$-string link. 
By \cite{MN}, we have that $\sigma$ is $C_2$-equivalent to a string link
\[\sigma{(0)}=\prod_{1\leq i<j\leq n}T(ij)^{\mu(ij)}.\]
So $\sigma$ is obtained from $\sigma(0)$ by surgery along linear $C_2$-trees.  
So by Lemmas~\ref{o-index} and \ref{string-knot}, there is a disjoint union $G_1$ of $C_2$-graphs with loops such that 
\begin{equation}\label{mTm1}
\sigma\stackrel{C_3}{\sim}\sigma{(0)}\cdot \sigma(1)\cdot (\1_n)_{G_1},
\end{equation}
where
\[\sigma{(1)}=\prod_{1\leq i\leq n} T(iii)^{\varepsilon(iii)}\cdot
\prod_{1\leq i< j\leq n} T(jij)^{\varepsilon(jij)}\cdot
\prod_{1\leq i<j<k\leq n} T(ijk)^{x(ijk)}\]
for some $\varepsilon(iii), \varepsilon(jij)\in\{0,1\}$ and some integers $x(ijk)$.  
Note that $T(iii)$, $T(jij)$ and $T(ijk)$ are the $n$-string links $T_i$, $W_{ji}$ and $B_{ijk}$ introduced in Subsection \ref{proofc3}, obtained respectively from $\1_n$ by surgery along the $C_2$-trees $t_i$, $w_{ji}$ and $b_{ijk}$ of Figure \ref{c2tree}.  In particular, we have $T(jij)\stackrel{C_3}{\sim} T(iji)$.  
By Lemma \ref{graph}, it follows that $\sigma\stackrel{C_3+c}{\sim}\sigma(0)\cdot\sigma(1)$.  

We denote by $\mathrm{Arf}_i(\sigma)$ the Arf invariant for (the closure of) the $i$th component of $\sigma$.
By Lemma~\ref{milnor-inv}, we have 
\[\begin{array}{rcl}
\mu_{\sigma}(ij) & = & \mu_{\sigma{(0)}}(ij), \\
\mathrm{Arf}_i(\sigma) & = & \mathrm{Arf}_i(\sigma{(0)})+\mathrm{Arf}_i(\sigma{(1)})=\mathrm{Arf}_i(\sigma(1))=\varepsilon(iii),\\
\mu_{\sigma}(ijk) & = & \mu_{\sigma{(0)}}(ijk)+\mu_{\sigma{(1)}}(ijk)=\mu_{\sigma{(0)}}(ijk)+x(ijk), \\
\mu_{\sigma}(jiij)&\equiv& \mu_{\sigma{(0)}}(jiij)+\mu_{\sigma{(1)}}(jiij)\\
&\equiv& \displaystyle\mu_{\sigma{(0)}}(jiij)+\sum_{1\leq i<j<k\leq n}x(ijk)\mu_{T(ijk)}(jiij)+
\varepsilon(jiij)~\text{mod~ $2$}.\end{array}\]
Since these invariants are $C_3$-concordance invariants, we have the following.
\begin{theorem}\label{c3c}
Two $n$-string links are $C_3$-concordant if and only if 
they share all invariants $\mathrm{Arf}_i$, $\mu(ij)~(1\leq i<j\leq n)$, 
$\mu(ijk)~(1\leq i<j<k\leq n)$ and $\mu(jiij)~\text{mod}~2~(1\leq i<j\leq n)$.
\end{theorem}

Since, by (\ref{mTm1}), the $n$-string link $\sigma$ is obtained from $\sigma(0)\cdot\sigma(1)\cdot(\1_n)_{G_1}$ by surgery along linear $C_3$-trees, by Lemmas~\ref{o-index} and \ref{string-knot}
, there is a 
disjoint union $G_2$ of $C_3$-graphs with loops such that 
\begin{equation}\label{mTm2}
\sigma\stackrel{C_4}{\sim}\sigma(0)\cdot\sigma(1)\cdot(\1_n)_{G_1}\cdot
\sigma(2)\cdot (\1_n)_{G_2},
\end{equation}
where 
\[\sigma(2):=\prod_{1\leq i< j\leq n} T(jiij)^{y(jiij)}\cdot
\prod_{\substack{ 1\leq i,j,k\leq n \\ i<j}} T(kijk)^{y(kijk)}\cdot
\prod_{1\leq i,j<k<l\leq n} T(kij)^{y(kijl)}\]
for some integers $y(jiij),y(kijk),y(kijl)$.  
Observe that $T(jiij)$, $T(kijk)$ and $T(kijl)$ correspond respectively to the string links obtained by surgery along the $C_3$-trees $h_{ij}$, $h_{ijk}$ and $h_{ijkl}$ of Figure \ref{c3tree}.  
(Actually, (\ref{mTm2}) can also be derived from the proof of Theorem \ref{c4}.)  

By Lemma~\ref{milnor-inv}, we have 
\[\mu_{\sigma}(jiij)=\mu_{\sigma(0)\cdot\sigma(1)}(jiij)+\mu_{\sigma(2)}(jiij)
=\mu_{\sigma(0)\cdot\sigma(1)}(jiij)+2y(jiij),\]
\[\mu_{\sigma}(kijk)= \mu_{\sigma(0)\cdot\sigma(1)}(kijk)+
\sum_{1\leq i<j\leq n}y(jiij)\mu_{T(jiij)}(kijk)+y(kijk),\]
\[\begin{array}{rcl}\mu_{\sigma}(kijl)&=& 
\displaystyle \mu_{\sigma(0)\cdot\sigma(1)}(kijl)
+\sum_{1\leq i<j\leq n}y(jiij)\mu_{T(jiij)}(kijl)\\
&&\displaystyle 
+\sum_{\substack{ 1\leq i,j,k\leq n \\ i<j}}y(kijk)\mu_{T(kijk)}(kijl)+y(kijl).\end{array}\]
Since these invariants are $C_4$-concordance invariants and 
$\sigma\stackrel{C_4+c}{\sim}\sigma(0)\cdot\sigma(1)\cdot\sigma(2)$, we have the 
following.
\begin{theorem}\label{c4c}
Two $n$-string links are $C_4$-concordant if and only if 
they are $C_3$-concordant and 
they share all invariants $\mu(jiij)~(1\leq i<j\leq n)$, 
$\mu(kijk)~(1\leq i,j,k\leq n~;~i<j)$ and 
$\mu(kijl)~(1\leq i,j<k<l\leq n)$.
\end{theorem}

Now, by(\ref{mTm2}), there is a disjoint union $G_3$ of $C_4$-graphs with loops such that 
\begin{equation}\label{mTm3}
\sigma\stackrel{C_5}{\sim}\sigma(0)\cdot\sigma_1\cdot(\1_n)_{G_1}\cdot
\sigma(2)\cdot (\1_n)_{G_2}\cdot\sigma(3)\cdot(\1_n)_{G_3},
\end{equation}
where $\sigma(3)$ is given by 
\begin{equation} \label{sigma3}
\sigma(3):=\displaystyle\prod_{i, j} T(jiiij)^{\varepsilon(jiiij)}\cdot
\prod_{i, j<k} T(kijik)^{\varepsilon(kijik)}\cdot
\prod_{s=1}^4\prod_{I\in {\mathcal I}_s}T(I)^{z(I)}
\end{equation}
for some $\varepsilon(jiiij)\in\{0,1\}$ and some integers $\varepsilon(kijik)$ and $z(I)$, 
where 
\[{\mathcal I}_1=\{ikkkj~|~1\leq i,j,k\leq n,~i<j\},\]
\[{\mathcal I}_2=\{kiijk~|~1\leq i,j,k\leq n,~,~i<j,~i<k\}\cup\{kjiik~|~1\leq j<i<k\leq n\},\]
\[{\mathcal I}_3=\{kijpk~|~1\leq i,j,k,p\leq n,~i<p\},~
{\mathcal I}_4=\{pijkq~|~1\leq i,j,k<p<q\leq n\}.\]

In particular, the second product in (\ref{sigma3}) is obtained from the following two observations.  
One one hand, for $1\le i, j<k\le n$, we have by (\ref{equU}) that
$$ T(ijkji)\stackrel{C_5}{\sim} T(kjijk)^{\delta}\cdot T(kijik)^{\delta'}~(\delta,\delta'\in\{-1,1\}) $$ 
(noting that $T(ijkij)$, $T(kjijk)$ and $T(kijik)$ correspond to the string links obtained by 
surgery along the $C_4$-trees $u''_{(12)}$, $u_{(12)}$ and $u'_{(12)}$ of Figure \ref{figc4_3} 
respectively, and that $u''_{(12)}\stackrel{C_5+c}{\sim} u''_{(132)}$).  
On the other hand, by Lemma \ref{o-index}~(3) the $C_4$-trees above are $2$-torsion elements in 
$\mathcal{SL}(n) / (C_{5}+c)$.  

Note also that $T(jiiij)$ in (\ref{sigma3}) corresponds to the string links obtained by surgery along 
the $C_4$-tree $s^k_{id}$ of Figure \ref{figc4_1} ($k=1,2$), and that similarly for 
$I\in \mathcal{I}_1,~\mathcal{I}_2$, the various $T(I)$ correspond to $C_5$-concordance 
classes of the string links obtained by surgery along the $C_4$-trees of Figure \ref{figc4_3}.  

By Lemma~\ref{milnor-inv}, we have 
\[\mu_{\sigma}(jiiiij)\equiv\mu_{\sigma(0)\cdot\sigma(1)\cdot\sigma(2)}(jiiiij)
+\varepsilon(jiiij)~\text{mod}~2,\]
\[\mu_{\sigma}(kijjik)\equiv\mu_{\sigma(0)\cdot\sigma(1)\cdot\sigma(2)}(kijjik)
+\sum_{i,j}\varepsilon(jiiij)\mu_{T(jiiij)}(kijjik)+
\varepsilon(kijik)~\text{mod}~2,\]
and for each $I\in{\mathcal I}_s~(1\leq s\leq 4)$ , 
\[\mu_{\sigma}(I)= \mu_{\sigma(0)\cdot\sigma(1)\cdot\sigma(2)}(I)
+\sum_{i,j}\varepsilon(jiiij)\mu_{T(jiiij)}(I)
+\sum_{\substack{W\in{\mathcal I}_t \\ t<s}}z(W)\mu_{T(W)}(I)+z(I).\]
Since these invariants are $C_5$-concordance invariants and 
$\sigma\stackrel{C_5+c}{\sim}\sigma(0)\cdot\sigma(1)\cdot\sigma(2)\cdot\sigma(3)$, we have the 
following.

\begin{theorem}\label{c5c}
Two $n$-string links are $C_5$-concordant if and only if they are $C_4$-concordant and share all invariants $\mu(jiiiij)~\text{mod}~2$ 
$(1\leq i,j\leq n)$, $\mu(kijjik)~\text{mod}~2$ $(1\leq~i,j<k\leq~n)$ and 
$\mu(I)~(I\in{\mathcal I}_1\cup{\mathcal I}_2\cup{\mathcal I}_3\cup{\mathcal I}_4)$.
\end{theorem}

Moreover, from (\ref{mTm3}) we have that there is
a string link $\sigma(4)$ and a disjoint union $G_4$ of $C_5$-graphs with loops 
such that 
\begin{equation}\label{pan}
\sigma\stackrel{C_6}{\sim}\sigma(0)\cdot\sigma(1)\cdot(\1_n)_{G_1}\cdot \sigma(2)\cdot (\1_n)_{G_2}\cdot\sigma(3)\cdot(\1_n)_{G_3}\cdot\sigma(4)\cdot(\1_n)_{G_4},
\end{equation}
where 
\[
\sigma(4)=\prod_{s=0}^7\prod_{J\in {\mathcal J}_s}T(J)^{w(J)}
\]
for some integers $w(J)$, where 
\[{\mathcal J}_0=\{ijijij~|~1\leq i<j\leq n\},~
{\mathcal J}_1=\{ijjjji~|~1\leq i,j\leq n\},\]
\[{\mathcal J}_2=\{ikkkkj~|~1\leq i,j,k\leq n,~i<j\},\]
\[{\mathcal J}_3=\{kiiijk,~kiijik,~kijjjk,~kjijjk~|~1\leq i,j,k\leq n,~i<j\},\]
\[{\mathcal J}_4=\{kiijjk, ~kijijk,~kijjik,~kjiijk|~1\leq i<j<k\leq n\},\]
\[\begin{array}{l}
{\mathcal J}_5=\{pikkjp,~pkijkp|~1\leq i,j,k,p\leq n,~i<j,~k<p\}\\
\hspace{6em}\cup\{pijkkp,~pikjkp~|~1\leq i,j<k<p\leq n\}\\
\hspace{6em}\cup\{pkkijp,~pkikjp~|~1\leq k<i,j,p\leq n\}\\
\hspace{6em}\cup\{pkkijp,~pkikjp~|~1\leq i,j,k,p\leq n,~i<k<j\},
\end{array}\]
\[{\mathcal J}_6=\{qkijpq~|~1\leq i,j,k,p,q\leq n,~k<p\},~{\mathcal J}_7=\{qijkpr~|~1\leq i,j,k,p<q<r\leq n\}.\]

Let us briefly explain how to determine these $\mathcal{J}_s$'s.  
First, separate $C_5$-trees by their indices. By Lemmas \ref{index1}, \ref{o-index}~(2) and \ref{graph}, 
we have eight cases : $\{i^{(3)},j^{(3)}\}$, $\{i^{(2)},j^{(4)}\}$, $\{i,j,k^{(4)}\}$, 
$\{i^{(3)},j,k^{(2)}\}$, $\{i^{(2)},j^{(2)},k^{(2)}\}$, 
$\{i,j,k^{(2)},p^{(2)}\}$, $\{i,j,k,p,q^{(2)}\}$ and $\{i,j,k,p,q,r\}$,  
which correspond to ${\mathcal J}_0$, ${\mathcal J}_1$, ${\mathcal J}_2$, 
${\mathcal J}_3$, ${\mathcal J}_4$, ${\mathcal J}_5$, ${\mathcal J}_6$ and ${\mathcal J}_7$ 
respectively. 
By the IHX relation, we may assume that each $C_5$-tree is linear, and we may chose any pair of leaves as ends.  
Hence for each of the eight cases above, we may choose the ends of any $C_5$-tree having the corresponding index.  Then we enumerate all possible o-indices, using Lemmas \ref{o-index}~(2) and \ref{graph}.  
For example, we may choose that the ends of any linear $C_5$-tree with index $\{i^{(3)},j^{(3)}\}$ are an $i$-leaf and a $j$-leaf, so we enumerate all o-indices starting with $i$ and ending with $j$.  By Lemma \ref{o-index}~(2), we are left with only two cases, namely $ijijij$ and $ijjiij$.  Now, it follows from two applications of the AS relation that $T(ijijij)\stackrel{C_6}{\sim} T(ijjiij)$.  So $ijijij$ is essentially the only o-index for $C_5$-trees with index $\{i^{(3)},j^{(3)}\}$.  

By combining a similar method as in \cite[Rem. 5.3]{akira} and the IHX relation, we have 
that for each $J\in{\mathcal J}_0$ 
\[\mu_{\sigma}(J)= \mu_{\sigma(0)\cdot\sigma(1)\cdot\sigma(2)\cdot\sigma(3)}(J)
+12\cdot w(J).\]
By Lemma~\ref{milnor-inv}, we have that
for each $J\in{\mathcal J}_s~(1\leq s\leq 7)$ 
\[\mu_{\sigma}(J)= \mu_{\sigma(0)\cdot\sigma(1)\cdot\sigma(2)\cdot\sigma(3)}(J)
+\sum_{0\leq t<s}\sum_{V\in{\mathcal J}_t}w(V)\mu_{T(V)}(J)+c_J\cdot w(J),\]
where $c_J=2$ if $J\in{\mathcal J}_1\cup \{kijjik,~kjiijk|~1\leq i<j<k\leq n\}$ 
and $c_J=1$ otherwise. 
Since these invariants are $C_6$-concordance invariants and 
$\sigma\stackrel{C_6+c}{\sim}\sigma(0)\cdot\sigma(1)\cdot\sigma(2)\cdot\sigma(3)\cdot\sigma(4)$, 
we have the following.
\begin{theorem}\label{c6c}
Two $n$-string links are $C_6$-concordant if and only if 
they are $C_5$-concordant and 
they share all invariants $\mu(J)$ for $J\in \mathcal{J}_i$ ($i=0,1,...,6$).   
\end{theorem}
\begin{remark}
Theorem \ref{c3c}, as well as the $2$-component cases of Theorems \ref{c4c}, \ref{c5c} and \ref{c6c}, 
are also proved in \cite{sase}, using different methods.  
\end{remark}
\begin{remark}
Note that for $k\le 6$, we meet new torsion elements in the group $\mathcal{SL}(n) / (C_k+c)$ for $k=3$ and $5$.  These are all $2$-torsion elements of the form $T(i_0,i_1,...,i_p,i_{p+1},i_p,...,i_1,i_0)$, possibly  with $i_j=i_k$ for $j,k\ne 0$.  By Lemma \ref{o-index}~(3), there are such 2-torsion elements in $\mathcal{SL}(n) / (C_k+c)$ for any odd $k$.  
\end{remark}
\section{$2$-string links up to self $C_3$-moves and concordance}
Given a multi-index $I$, let $r(I)$ denote the maximum number of times that any index appears. For example, $r(1123)=2,~r(1231223)=3$. It is known that if $r(I)=1$, then Milnor invariant with index $I$ is a {\em link-homotopy} invariant \cite{Milnor,HL}, where link-homotopy is an equivalence relation on links generated by self crossing changes. 
Milnor invariants give a link-homotopy classification of string links \cite{HL}. 
  
Although Milnor invariants with $r\geq 2$ are not necessarily link-homotopy invariants, Fleming and the second author showed that $\mu$-invariants with $r\leq k$ are {\em self $C_k$-equivalence} invariants for string links, where the self $C_k$-equivalence is an equivalence relation on (string) links generated by self $C_k$-moves, which are $C_k$-moves with all $k+1$ strands in a single component.  See \cite[Theorem 3.1]{FY} and \cite{yasuhara}.  

Two string links $\sigma$ and $\sigma'$ are {\em self-$C_k$ concordant} if there is a sequence $\sigma=\sigma_1,...,\sigma_m=\sigma'$ of string links such that for each $i(\in\{1,...,m-1\}$), $\sigma_i$ and $\sigma_{i+1}$ are either concordant or self $C_k$-equivalent. 

Since Milnor invariants are concordance invariants, any Milnor invariant indexed by $I$ with $r(I)\leq k$ is a self-$C_k$ concordance invariant.  The second author showed that Milnor invariants $\mu(I)$ with $r(I)\leq 2$ classify string links up to self-$C_2$ concordance \cite{yasuhara}. 
Here we give a self-$C_3$ concordance classification for $2$-string links. 
\begin{theorem}
Two $2$-string links are self-$C_3$ concordant of and only if 
they share all invariants $\mathrm{Arf}_i~(i=1,2)$, $\mu(12)$, 
$\mu(2112)$, $\mu(121212)$, and $\mu(jiiiij)$ mod $2$ $(\{i,j\}=\{1,2\})$.
\end{theorem}
\begin{remark}
In \cite[Remark]{yasuhara}, the second author asked if the Hopf link with both components Whitehead doubled is self-$C_3$ equivalent to the trivial $2$-string link.  The theorem above gives an affirmative answer.
\end{remark}
\begin{proof}
By \cite[Lem. 1.2]{FY}, $C_6$-concordance implies self-$C_3$ concordance for $2$-string links.  Starting with a representative of the $C_6$-concordance class of a $2$-string link $\sigma$, as given by (\ref{pan}), we can eliminate the generators $T(I)$ such that $I$ contains at least 3 times the same index to obtain a self-$C_3$ concordance representative. We obtain that $\sigma$ is self-$C_3$ concordant to 
\[T(12)^x\cdot T(111)^{\varepsilon_1}\cdot
T(222)^{\varepsilon_2}\cdot T(212)^{\varepsilon_3}\cdot 
T(2112)^{y}\cdot T(21112)^{\varepsilon_4}
T(12221)^{\varepsilon_5}\cdot T(121212)^{z},\]
for some integers $x,y,z$ and for some $\varepsilon_i\in\{0,1\}~(1\le i\le 5)$.  
By Theorem~\ref{c6c}, $x,y,z, \varepsilon_i~(1\le i\le 5)$ are determined by the invariants $\mathrm{Arf}_i~(i=1,2)$, $\mu(12)$, $\mu(2112)$, $\mu(121212)$, $\mu(211112)$ mod $2$, and $\mu(122221)$ mod $2$.
Since all these invariants except for mod-$2$ $\mu(jiiiij)$ ($\{i,j\}=\{1,2\}$) are self-$C_3$ concordance invariants, there only remains to show that mod-$2$ $\mu(jiiiij)$ is a self-$C_3$ equivalence invariant. 

Suppose that $\sigma'$ is a string link obtained from $\sigma\in \mathcal{SL}(2)$ by surgery along a $C_3$-tree with index $\{1^{(4)}\}$. It is enough to show that $\mu_{\sigma}(jiiiij)\equiv \mu_{\sigma'}(jiiiij)$ mod 2.  
By Calculus of Clasper, 
\[\sigma'\stackrel{C_4}{\sim}\sigma\cdot(\1_2)_{t'},\]
where $t'$ is a $C_3$-tree with index $\{1^{(4)}\}$ and is in a tubular neighbourhood of the $1$st strand of $\1_2$.
By \cite[Lem.~2.1]{akira}, we may assume that the $C_4$-equivalence above is realized by surgery along a disjoint union of $C_4$-trees with indices $\{1^{(5)}\}$ or $\{1^{(4)},2\}$.
So by Lemma \ref{o-index}~(2), we have
\[\sigma'\stackrel{C_5+c}{\sim} \sigma\cdot(\1_2)_{t'}.\]
Hence by Remark \ref{hop} we have 
$$\mu_{\sigma'}(jiiiij)\equiv\mu_{\sigma}(jiiiij)+\mu_{(\1_2)_{t'}}(jiiiij)\equiv \mu_{\sigma}(jiiiij)~\text{mod}~2.$$
This completes the proof.  
\end{proof} 
\section{$C_{n+1}$-moves for $n$-component Brunnian string links} \label{BSL}
An $n$-string link is Brunnian if every proper substring link of it is trivial.  
In this section, we use tools developped in the present paper to classify Brunnian $n$-string links up to $C_{n+1}$ equivalence, thus improving a previous result of the authors \cite{MY}.   

Let $B$ be a Brunnian $n$-string link.  
An explicit formula for a representative $B_0$ of the $C_n$-equivalence class of $B$ was given in \cite{HM} (see also \cite[Prop. 4.2]{MY}), and can be formulated as follows (using the notation of section 5): 
\begin{equation} \label{b0}
 B_0:=\prod_{\eta\in S_{n-2}} T(n-1,\eta(1),...,\eta(n-2),n)^{\mu_B(n-1,\eta(1),...,\eta(n-2),n)}.  
\end{equation}

Recall from Subsection \ref{calculus} that for an integer $k$ in $\n$, $\mathcal{B}_n(k)$ denotes the set 
of all bijections $\tau$ from $\{ 1,...,n-1 \}$ to $\n \setminus \{k \}$ such that $\tau(1)<\tau(n-1)$, 
and that $B_{\alpha}(k)$, resp. $\overline{B_{\alpha}}(k)$, is the $n$-string link obtained from $\1_n$ 
by surgery along the $C_n$-tree $T_{\alpha}(l)$, resp.  $\overline{T_{\alpha}}(l)$ represented 
in Figure \ref{flipping_fig}.  
For $\tau\in \mathcal{B}_n(k)$, set $\mu_{\tau}(B) := \mu_B(\tau(1),...,\tau(n-1),k,k)$.
Is was proved in \cite[Prop. 4.5]{MY} that 
\[B \stackrel{C_{n+1}}{\sim}  B_0\cdot B_{(1)}\cdot ... \cdot B_{(n)},\]
where, for each $k~(1\le k\le n)$, $B_{(k)}$ is the Brunnian $n$-string link
\begin{equation}\label{equBn}
 \prod_{\tau\in \mathcal{B}_n(k)} (B_{\tau}(k))^{n_{\tau}(k)}\cdot (\overline{B_{\tau}}(k))^{n'_{\tau}(k)}, 
\end{equation}
such that, for any $\tau\in \mathcal{B}_n(k)~(1\le k\le n)$, the exponents $n_{\tau}(k)$ and $n'_{\tau}(k)$ are two 
integers satisfying 
\[  n_{\tau}(k)+n'_{\tau}(k)=\mu_{\tau}(B_{(1)}\cdot ... \cdot B_{(n)})=\mu_{\tau}(B)-\mu_{\tau}(B_0). \]
This uses the fact that, for any $k\in \n$ and  $\alpha,\tau \in \mathcal{B}_4(k)$, we have 
 $$\mu_\alpha (B_{\tau}(k))=\mu_{\alpha} (\ov{B_{\tau}}(1))=\delta_{\alpha,\tau}.  $$

Given an $n$-string link $\si$ and $\tau\in \mathcal{B}_n(1)$, we can construct a knot $K_\tau(\si)$ in $S^3$ as follows. 
Connect the upper endpoints of the first and the ${\tau(1)}$th components of $\si$ by an arc $a_1$ 
in $S^3\setminus (D^2\times I)$.  
Next, connect the lower endpoints of the ${\tau(1)}$th and the ${\tau(2)}$th components 
by an arc $a_2$ in $S^3\setminus (D^2\times I)$ disjoint from $a_1$, then the upper endpoints of the ${\tau(2)}$th and ${\tau(3)}$th components by an arc $a_3$ in $S^3\setminus (D^2\times I)$ disjoint from 
$a_1\cup a_2$.  Repeat this construction until reaching the ${\tau(n-1)}$th component, and connect its lower or upper endpoint (depending on the parity of $n$) to the lower enpoint of the first component by an arc $a_n$ in $S^3\setminus (D^2\times I)$ 
disjoint from $\bigcup_{1\le i\le n-1} a_i$.  The arcs are chosen so that, if $a_i$ and $a_j$ ($i<j$) meet in the 
diagram of $L$, then $a_i$ overpasses $a_j$.  
It follows from the construction of $K_\tau(\si)$ and \cite{horiuchi} that for any $\tau\in \mathcal{B}_n(1)$, we have $\mathbf{p}:=P^{(n)}_0(K_\tau(B_\tau(1)))$ is nonzero (note that $\mathbf{p}$ depends only on $n$).  
Set 
\[ f_\tau(\si):= P^{(n)}_0(K_\tau(\si)) / \mathbf{p}. \]

By the proof of Lemma \ref{flipping}, 
we note that for any $k\geq 2$ there is a bijection 
$\phi_k:\mathcal{B}_n(1)\longrightarrow\mathcal{B}_n(k)$ such that  
for any $\tau\in\mathcal{B}_n(1)$, we have  
${B}_{\tau}(1)\cdot\overline{{B}_{\tau}}(1)^{-1}
\stackrel{C_{n+1}}{\sim}{B}_{\phi_k(\tau)}(k)\cdot\overline{{B}_{\phi_k(\tau)}}(k)^{-1}$.

We can now prove the following stronger version of \cite[Prop. 4.5]{MY}.  
\begin{theorem}
Let $B$ be a Brunnian $n$-string link.  Then 
$$ B\stackrel{C_{n+1}}{\sim} B_0\cdot B', $$
where $B_0$ is determined by the Milnor invariants of $B$ of length $n$ as in (\ref{b0}), 
and where $B'$ is given by 
\[ \prod_{\tau\in \mathcal{B}_n(1)} 
\left((B_{\tau}(1))^{m_{\tau}} \cdot 
(\overline{B_{\tau}}(1))^{\mu_{\tau}(B)-\mu_{\tau}(B_0)-m_{\tau}}\right)\cdot 
\prod_{k=2}^n\prod_{\tau\in \mathcal{B}_n(k)} ({B_{\tau}}(k))^{\mu_{\tau}(B)-\mu_{\tau}(B_0)}, 
\] 
where $m_{\tau}={f_{\tau}(B)-f_{\tau}(B_0)-
\sum_{k\geq 2}(\mu_{\phi_k(\tau)}(B)-\mu_{\phi_k(\tau)}(B_0))}$ ($\tau\in \mathcal{B}_n(1)$).  
\end{theorem}
\begin{proof}
By Lemma \ref{flipping}, we may assume that $n'_{\tau}(k)=0$ in (\ref{equBn}) 
for any $\tau\in \mathcal{B}_n(k)$ with $k\ne 1$. 
Hence the product $B_{(1)}\cdot ... \cdot B_{(n)}$ is given by  
\[ \prod_{\tau\in \mathcal{B}_n(1)} \left((B_{\tau}(1))^{n_{\tau}(1)}\cdot 
(\overline{B_{\tau}}(1))^{n'_{\tau}(1)}\right)\cdot
\prod_{k=2}^n\prod_{\tau\in \mathcal{B}_n(k)} (B_{\tau}(k))^{\mu_{\tau}(B)-\mu_{\tau}(B_0)}.\]

Let $\tau\in \mathcal{B}_n(1)$.  It follows from the construction of $K_\tau(\si)$ and \cite{horiuchi} that for 
any $\eta\in\mathcal{B}_n(1)$, we have $f_\tau(B_{\eta}(1))= \delta_{\tau,\eta}$ and 
$f_\tau(\overline{B_{\eta}}(1))= 0$, 
and that $f_\tau(B_{\phi_k(\eta)}(k))= \delta_{\tau,\eta}$ for each $k\geq 2$.  
%

By using a similar argument as Claim \ref{claimadditivity} and the multiplicativity of 
the HOMFLYPT polynomial, we thus have that for each $\tau\in \mathcal{B}_n(1)$ 
\begin{displaymath}
\begin{array}{rcl}f_{\tau}(B) &=& f_{\tau}(B_0)+f_{\tau}(B_{(1)}\cdots B_{(n)})\\
&=&
\displaystyle f_{\tau}(B_0)+n_{\tau}(1)
+\sum_{k\geq 2}(\mu_{\phi_k(\tau)}(B)-\mu_{\phi_k(\tau)}(B_0)).  
\end{array}\end{displaymath}
Since
$\mu_{\tau}(B)=\mu_{\tau}(B_0)+n_{\tau}(1)+n'_{\tau}(1)$,
this completes the proof.  
\end{proof}

%

\begin{thebibliography}{99}
\bibitem{BNv} D. Bar-Natan, \emph{On the Vassiliev knot invariants}, Topology \textbf{34} (1995), 423--472.

\bibitem{BN} D. Bar-Natan, {\it Vassiliev homotopy string link invariants}, J. Knot Theory Ram. {\bf 4}, no. 1 (1995), 13--32.

\bibitem{BNcalc} D. Bar-Natan, \emph{Some computations related to Vassiliev invariants} (1996), available at 
                   \texttt{http://www.math.toronto.edu/}$\sim$\texttt{drorbn/}.

\bibitem{casson} A.J. Casson, \emph{Link cobordism and Milnor's invariant}, Bull. London Math. Soc. \textbf{7} (1975), 39–-40.

\bibitem{cochran} T.D. Cochran, {\it Derivatives of link: Milnor's concordance invariants and Massey's products}, Mem. Amer. Math. Soc. {\bf 84} (1990), No. 427.

\bibitem{CT} J. Conant, P. Teichner, \emph{Grope cobordism of classical knots}, Topology \textbf{43} (2004), 119--156.  

\bibitem{FY} T. Fleming and A. Yasuhara, {\it Milnor's invariants and self $C_k$-equivalence}, Proc. Amer. Math. Soc. {\bf 137} (2009) 761-770. 

\bibitem{GL} S. Garoufalidis, J. Levine, \emph{Concordance and 1-loop clovers}, Alg. Geom. Topol. \textbf{1} (2001), 687--697. 

\bibitem{Gusarov:91} M.N. Gusarov, {\it A new form of the Conway-Jones polynomial of oriented links}. (Russian), Zap. Nauchn. Sem. Leningrad. Otdel. Mat. Inst. Steklov. (LOMI) {\bf 193} (1991), Geom. i Topol. {\bf 1}, 4--9, 161; translation in ``Topology of manifolds and varieties'', 167--172, Adv. Soviet Math., 18, Amer. Math. Soc., Providence, RI, 1994.

\bibitem{Gusarov:94} M. Gusarov, {\it On $n$-equivalence of knots and invariants of finite degree}, ``Topology of manifolds and varieties'', 173--192, Adv. Soviet Math., 18, Amer. Math. Soc., Providence, RI, 1994.

\bibitem{G} M.N. Gusarov, {\it Variations of knotted graphs. The geometric technique of $n$-equivalence}. (Russian), Algebra i Analiz {\bf 12} (2000), no. 4, 79--125; translation in St. Petersburg Math. J. {\bf 12} (2001), no. 4, 569--604.
  
\bibitem{HL} N. Habegger, X.S. Lin, {\em The classification of links up to link-homotopy}, J. Amer. Math. Soc. {\bf 3} (1990), 389--419.

\bibitem{HMa} N. Habegger, G. Masbaum, {\it The Kontsevich integral and Milnor's invariants}, Topology {\bf 39} (2000), no. 6, 1253--1289.

\bibitem{H} K. Habiro, {\it Claspers and finite type invariants of links}, Geom. Topol. {\bf 4} (2000), 1--83.

\bibitem{Hgraph} K. Habiro, {\it Replacing a graph clasper by tree claspers}, preprint math.GT/0510459.

\bibitem{HM} K. Habiro, J.B. Meilhan, {\it Finite type invariants and 
  Milnor invariants for Brunnian links}, Int. J. Math. \textbf{19}, no. 6 (2008), 747--766.  

\bibitem{horiuchi} S. Horiuchi, {\it The Jacobi diagram for a $C_n$-move and the HOMFLY polynomial}, J. Knot Theory Ram. \textbf{16}, no. 2 (2007), 227--242.  

\bibitem{kanenobu} T. Kanenobu, \emph{Finite type invariants of order $4$ for $2$-component links}, from: ``Intelligence of Low Dimensional Topology 2006'', (J.S. Carter \emph{et al} editors), Ser. Knots Everything \textbf{40}, World Sci. Publ., Hackensack, NJ (2007), 109--116.  

\bibitem{KM} T. Kanenobu, Y. Miyazawa, {\it HOMFLY polynomials as Vassiliev 
  link invariants}, in \emph{Knot theory}, Banach Center Publ. \textbf{42}, 
  Polish Acad. Sci., Warsaw (1998,) 165--185.

\bibitem{kmt} T. Kanenobu, Y. Miyazawa and A. Tani, \emph{Vassiliev link invariants of order three}, J. Knot Theory
Ram. \textbf{7} (1998), 433-–462.

\bibitem{Kontsevich} M. Kontsevich, {\it Vassiliev's knot invariants},
  ``I. M. Gel'fand Seminar'', 137--150, Adv. Soviet Math., 16, Part 2,
  Amer. Math. Soc., Providence, RI, 1993.

\bibitem{lickorish} W.B.R. Lickorish, \emph{An Introduction to Knot Theory}, GTM 175, Springer-Verlag, New York 1997. 

\bibitem{Lin} X.S. Lin, {\it Power series expansions and invariants of
  links}, in ``Geometric topology'', AMS/IP Stud. Adv. Math. 2.1,
  Amer. Math. Soc. Providence, RI (1997) 184--202.

\bibitem{massuyeau} G. Massuyeau, \emph{Finite-type invariants of three-manifolds and the dimension subgroup problem}, J. London Math. Soc. \textbf{75:3} (2007), 791--811.

\bibitem{jbjktr} J.B. Meilhan, \emph{On Vassiliev invariants of order two for string links}, J. Knot Theory Ram. \textbf{14} (2005), No. 5, 665--687. 

\bibitem{MY} J.B. Meilhan, A. Yasuhara, \emph{On Cn-moves for links}, Pacific J. Math. \textbf{238} (2008), 119--143. 

\bibitem{Milnor} J. Milnor, {\it Link groups}, Ann. of Math. (2) {\bf 59} (1954), 177--195.

\bibitem{Milnor2} J. Milnor, {\it Isotopy of links}, Algebraic geometry and topology, 
      A symposium in honor of S. Lefschetz, pp. 280--306, 
      Princeton University Press, Princeton, N. J., 1957. 

\bibitem{MN} H. Murakami, Y. Nakanishi, \emph{On a certain move generating link-homology}, Math. Ann. , 
            $\mathbf{283}$ (1989), 75--89. 

\bibitem{ng} K.Y. Ng, \emph{Groups of ribbon knots}, Topology \textbf{37} (1998), 441--458.   

\bibitem{OY} Y. Ohyama, H. Yamada, \emph{A $C_n$-move for a knot and the coefficients of the Conway polynomial}, J. Knot Theory Ram. \textbf{17} (2008),  no. 7, 771--785. 

\bibitem{rob} R.A. Robertello, \emph{An invariant of knot cobordism}, Comm. Pure Appl. Math. \textbf{18} (1965), 543-555.

\bibitem{sase} T. Sase, \emph{$C_k$-concordant ni yoru string link no bunrui}, Master's thesis, Tokyo Gakugei Univ. (2009).	

\bibitem{Vassiliev} V.A. Vassiliev, {\it Cohomology of knot spaces}, ``Theory of singularities and its applications'', 23--69, Adv. Soviet Math., 1, Amer. Math. Soc., Providence, RI, 1990.

\bibitem{yasuhara} A. Yasuhara, \emph{Classification of string links up to self delta-moves and concordance}, Alg. Geom. Topol. \textbf{9} (2009), 265-–275.   

\bibitem{akira} A. Yasuhara, \emph{Self Delta-equivalence for Links Whose Milnor's Isotopy Invariants Vanish}, to appear in Trans. Amer. Math. Soc.   

\end{thebibliography}
\end{document}